\documentclass{amsart}

\usepackage{enumerate,amssymb,amsmath,mathrsfs}
\usepackage{amssymb,amsfonts,amsthm,amscd}
\usepackage[mathscr]{eucal}     
\usepackage{graphicx}
\usepackage[arrow, matrix, curve]{xy}
\usepackage{color}
\definecolor{gray}{gray}{.75}
\definecolor{gray2}{gray}{.50}

\newcommand{\dom}{\mathcal{D}}

\newcommand{\Z}{\mathbb{Z}}
\newcommand{\N}{\mathbb{N}}
\newcommand{\R}{\mathbb{R}}
\newcommand{\C}{\mathbb{C}}

\newcommand{\A}{\mathrm{\alpha}}

\newcommand{\la}{\mathrm{\lambda}}

\newtheorem{thm}{Theorem}[section]
\newtheorem{cor}[thm]{Corollary}
\newtheorem{remark}[thm]{Remark}
\newtheorem{prop}[thm]{Proposition}
\newtheorem{lemma}[thm]{Lemma}
\newtheorem{defn}[thm]{Definition}

\numberwithin{equation}{section}

\begin{document}

\title[Analytic Torsion]
{Analytic Torsion of a Bounded Generalized Cone}
\author{Boris Vertman}
\address{University of Bonn \\ Department of Mathematics \\
Beringstr. 6\\ 53115 Bonn\\ Germany}
\email{vertman@math.uni-bonn.de}


\begin{abstract}
{We compute the analytic torsion of a bounded generalized cone by generalizing the computational methods of M. Spreafico and using the symmetry in the de Rham complex, as established by M. Lesch. We evaluate our result in lower dimensions and further provide a separate computation of analytic torsion of a bounded generalized cone over $S^1$, since the standard cone over the sphere is simply a flat disc.}
\end{abstract}

\maketitle

\pagestyle{myheadings}
\markboth{\textsc{Analytic Torsion}}{\textsc{Boris Vertman}} 

\section{Introduction}\
\\[-3mm] Torsion invariants for manifolds which are not simply connected were introduced by K. Reidemeister in [Re1, Re2] and generalized to higher dimensions by W. Franz in [Fr]. Using the introduced torsion invariants the authors obtained a full PL-classification of lens spaces. The Reidemeister-Franz torsion, short $-$ the Reidemeister torsion, was the first invariant of manifolds which was not a homotopy invariant.
\\[3mm] The Reidemeister-Franz definition of torsion invariants was extended later to smooth manifolds by J. H. Whitehead in [Wh] and G. de Rham in [Rh]. With their construction G. de Rham further proved that a spherical Clifford-Klein manifold is determined up to isometry by its fundamental group and its Reidemeister torsion. 
\\[3mm] The Reidemeister-Franz torsion is a combinatorial invariant and can be constructed using a cell-decomposition or a triangulation of the underlying manifold. The combinatorial invariance under subdivisions was established by J. Milnor in [Mi], see also [RS]. It is therefore a topological invariant of $M$, however not a homotopy invariant.
\\[3mm] There is a series of results relating combinatorial and analytic objects, among them the Atiyah-Singer Index Theorem. In view of these results it is natural to ask for the analytic counterpart of the combinatorial Reidemeister torsion. Such an analytic torsion was introduced by D. B. Ray and I. M. Singer in [RS] in form of a weighted product of zeta-regularized determinants of Laplace operators on differential forms.
\\[3mm] The zeta-regularized determinant of a Laplace Operator is a spectral invariant which very quickly became an object of interest on its own in differential and conformal geometry, studied in particular as a function of metrics for appropriate geometric operators. Further it plays a role in mathematical physics where it gives a regularization procedure of functional path integrals (partition function), see [H]. 
\\[3mm] In their work D.B. Ray and I. M. Singer provided some motivation why the analytic torsion should equal the combinatorial invariant. The celebrated Cheeger-M\"{u}ller Theorem, established independently by J. Cheeger in [Ch] and W. M\"{u}ller in [Mu1], proved equality between the analytic Ray-Singer torsion and the combinatorial Reidemeister torsion for any smooth closed manifold with an orthogonal representation of its fundamental group.
\\[3mm] The proofs of J. Cheeger and W. M\"{u}ller use different approaches. The first author in principle studied the behaviour of the Ray-Singer torsion under surgery. The second author used combinatorial parametrices and approximation theory of Dodziuk [Do] to reduce the problem to trivial representations, treating this problem then by surgeries.
\\[3mm] Note a different approach of Burghelea-Friedlander-Kappeler in [BFK] and Bismut-Zhang in [BZ1], who obtained a new proof of the result by J. Cheeger and W. M\"{u}ller using Witten deformation of the de Rham complex via a Morse function.
\\[3mm] The study of the analytic torsion of Ray and Singer has taken the following natural steps. The setup of a closed Riemannian manifold with its marking point $-$ the Cheeger M\"{u}ller Theorem, was followed by the discussion of compact manifolds with smooth boundary. In the context of smooth compact manifolds with boundary a Cheeger-M\"{u}ller type result was established in the work of W. L\"{u}ck [L\"{u}] and S. Vishik [V]. 
\\[3mm] While the first author reduced the discussion to the known Cheeger-M\"{u}ller Theorem on closed manifolds via the closed double construction, the second author gave an independent proof of the Cheeger-M\"{u}ller Theorem on smooth compact manifolds with and without boundary by establishing gluing property of the analytic torsion. 
\\[3mm] Both proofs work under the assumption of product metric structures near the boundary. However by the anomaly formula in [BM] the assumption of product metric structures can be relaxed. 
\\[3mm] The next natural step in the study of analytic torsion is the treatment of Riemannian manifolds with singularities. We are interested in the simplest case, the conical singularity. The analysis and the geometry of spaces with conical singularities were developped in the classical works of J. Cheeger in [Ch1] and [Ch2]. This setup is modelled by a bounded generalized cone $M=(0,1]\times N$ over a closed Riemannian manifold $(N,g^N)$ with the Riemannian metric
$$g^M=dx^2 \oplus x^2g^N.$$
The analytic Ray-Singer Torsion is shown in [Dar] to exist on a bounded generalized cone and the natural question arises whether one can establish a Cheeger-M\"uller type theorem in the singular setup, as well. Following [L, Problem 5.3], the idea is to reduce via the gluing formula of Vishik [V] the comparison of Ray-Singer and $\bar{p}$-Reidemeister torsion (intersection torsion, cf. [Dar]) on compact manifolds with conical singularities to a comparison on a bounded generalized cone. 
\\[3mm] The presented computation of analytic torsion on a bounded generalized cone solves problem posed in [L, Problem 5.3]. We have provided the general answer to the question in Theorems \ref{final-odd} and \ref{final-even} and obtained as an example explicit results in two and in three dimensions in Corollaries \ref{two-dim} and \ref{three-dim}. 
\\[3mm] After the computation of the analytic torsion of a bounded generalized cone one faces the problem of comparing it to the intersection torsion in the "right" perversity $\bar{p}$. However the complex form of the result for the analytic torsion at least complicates the comparison with the topological counterpart. 
\\[3mm] In the actual computation of the analytic torsion of a bounded generalized cone, we use the approach of M. Spreafico [S], combined with elements of [BKD], together with an observation of symmetry in the de Rham complex by M. Lesch in [L3]. 
\\[3mm] In fact J.S. Dowker and K. Kirsten provided in [DK] some explicit results for a bounded generalized cone $M=(0,1]\times N$, giving formulas which related the zeta-determinants of form-valued Laplacians, essentially self-adjoint at the cone singularity and with Dirichlet or generalized Neumann conditions at the cone base, to the spectral information on the base manifold $N$. So, in the manner of [Ch2], they reduced analysis on the cone to that on its base.  
\\[3mm] Theoretically these results can be composed directly into a formula for the analytic torsion. However this approach would disregard the subtle symmetry of the de Rham complex of a bounded generalized cone, which was derived by M. Lesch in [L3]. Furthermore the formulas obtained this way turn out to be rather ineffective.
\\[3mm] We present here an approach that does make use of the symmetry of the de Rham complex and leads to expressions that are easier to evaluate. The calculations are performed for any dimension $\geq 2$ with an overall general result for the analytic torsion of a bounded generalized cone. Further calculations are possible only by specifying the base manifold $N$. In Subsection \ref{an-torsion-general} we provide explicit results in three and in two dimensions.
\\[3mm] The computation is performed for simplicity under an additional assumption of a scaled metric $g^M$, such that the form-valued Laplacians are essentially self-adjoint at the cone singularity. This apparent gap can be considered as closed by the discussion in [BV1]. 
\\[3mm] For a bounded generalized cone of dimension two, over a one-dimensional sphere, one needs to introduce an additional parameter in the Riemannian cone metric in order to deal with bounded generalized cone and not simply with a flat disc $D^1$. There is no need to evaluate the symmetry of the de Rham complex in this case. The calculations of [S] can be generalized to this setup in a straightforward way, which is done in Section \ref{an-torsion-sphere}.
\\[3mm] {\bf Acknowledgements.} The results of this article were obtained during the author's Ph.D. studies at Bonn University, Germany. The author would like to thank his thesis advisor Prof. Matthias Lesch for his support and useful discussions. The author was supported by the German Research Foundation as a scholar of the Graduiertenkolleg 1269 "Global Structures in Geometry and Analysis".

\section{The Scalar Analytic Torsion}\
\\[-3mm] Consider a bounded generalized cone $M=(0,1]\times N$ over a closed oriented Riemannian manifold $(N,g^N)$ of dimension $\dim N =n$, with the Riemannian metric on $M$ given by a warped product $$g^M = dx^2 \oplus x^2g^N.$$ 
Consider the exterior derivatives and their formal adjoints on differential forms with compact support in the interior of $M$:
\begin{align*}
d_k:\Omega^k_0(M)\rightarrow \Omega_0^{k+1}(M), \\
d^t_k: \Omega^{k+1}_0(M)\rightarrow \Omega^k_0(M).
\end{align*}
Define the minimal closed extensions $d_{k,\min}$ and $d^t_{k,\min}$ as the graph closures in $L^2(\bigwedge\nolimits^*T^*M,\textup{vol}(g^M))$ of the differential operators $d_k$ and $d^t_k$ respectively. The operators $d_{k,\min}$ and $d^t_{k,\min}$ are closed and densely defined. In particular we can form the adjoint operators and set for the maximal extensions:
\begin{align*}
d_{k,\max}:=(d^t_{k,\min})^*, \quad d^t_{k,\max}:=(d_{k,\min})^*.
\end{align*}
The minimal and the maximal extensions of the exterior derivative give rise to self-adjoint extensions of the associated Laplace operator $\triangle_k$. We are interested in the relative self-adjoint extension of $\triangle_k$, defined as follows:
\begin{align}\label{definition}
\triangle_k^{rel}:= &\, d^*_{k,\min}d_{k,\min}+d_{k-1,\min}d^*_{k-1,\min}= \\
=&\, d^t_{k,\max}d_{k,\min}+d_{k-1,\min}d^t_{k-1,\max}. \nonumber 
\end{align}
We will only need the following well-known result, which is a direct application of [L1, Proposition 1.4.7]
\begin{thm}\label{cone-zeta-function}
The self-adjoint operator $\triangle_k^{rel}$ is discrete with the zeta-function $$\zeta(s,\triangle_k^{rel})=\sum\limits_{\lambda \in Sp(\triangle_k^{rel})\backslash \{0\}}\lambda^{-s}, \ Re(s)>m/2,$$ being holomorphic for $Re(s)>m/2$. 
\end{thm}\ \\
\\[-7mm] The meromorphic continuation of zeta-functions for general self-adjoint extensions of regular-singular operators is discussed in a series of sources, notably [L1, Theorem 2.4.1], [Ch2, Theorem 4.1] and [LMP, Theorem 5.7]. 
\\[3mm] For a compact oriented Riemannian manifold $X^m$ the scalar analytic torsion ([RS])
is defined by  $$\log T(X):=\frac{1}{2}\sum_{k=0}^m(-1)^k\cdot k \cdot \zeta'_{k}(0),$$ where $\zeta_k(s)$ denotes the zeta-function of the Laplacian on k-forms of $X$, with relative or the absolute boundary conditions posed at $\partial X$. On compact Riemannian manifolds the zeta-functions $\zeta_k(s)$ extend meromorphically to $\C$ with $s=0$ being regular, so the definition makes sense.
\\[3mm] On the bounded generalized cone $M$ the zeta-functions $\zeta(s,\triangle_k^{rel})$ possibly have a simple pole at $s=0$. However we have the following result of A.Dar:
\begin{thm} \textup{[Dar]} The meromorphic function 
\begin{align}\label{T-function} 
T(M,s):=\frac{1}{2}\sum\limits_{k=0}^m(-1)^k\cdot k\cdot \zeta(s,\triangle_k^{rel})
\end{align}
is regular at $s=0$. Thus the analytic torsion $T^{rel}(M):=\exp (T'(M,s=0))$ of a bounded generalized cone exists.
\end{thm} \ \\
\\[-7mm] Thus even though the zeta-functions $\zeta(s,\triangle_k^{rel})$ need not be regular at $s=0$, their residua at $s=0$ cancel in the alternating weighted sum $T(M,s)$.
\\[3mm] The statement extends to general compact manifolds with isolated conical singularities. A compact manifold with a conical singularity is a Riemannian manifold $(M_1\cup_N U, g)$ partitioned by a compact hypersurface $N$, such that $M_1$ is a compact manifold with boundary $N$ and $U$ is isometric to $(0,\epsilon]\times N$ with the metric over $U$ being of the following form $$g|_U=dx^2\oplus x^2g|_N.$$
In this article we compute for the bounded generalized cone $M$ the analytic continuation of $\log T(M,s)$ to $s=0$ by means of a decomposition of the de Rham complex. We perform the computations under an isometric identification below.
\\[3mm] The volume forms on $M$ and $N$, associated to the Riemannian metrics $g^M$ and $g^N$, are related as follows: 
$$\textup{vol}(g^M)=x^n dx \wedge \textup{vol}(g^N).$$
Consider as in [BS, (5.2)] the following separation of variables map, which is linear and bijective:
\begin{align}\label{separation}
\Psi_k : C^{\infty}_0((0,1),\Omega^{k-1}(N)\oplus \Omega^k(N))\to \Omega_0^k(M) \\
(\phi_{k-1},\phi_k)\mapsto x^{k-1-n/2}\phi_{k-1}\wedge dx + x^{k-n/2}\phi_k, \nonumber
\end{align}
where $\phi_k,\phi_{k-1}$ are identified with their pullback to $M$ under the natural projection $\pi: (0,1]\times N\to N$ onto the second factor, and $x$ is the canonical coordinate on $(0,1]$. Here $\Omega_0^k(M)$ denotes differential forms of degree $k=0,..,n+1$ with compact support in the interior of $M$. The separation of variables map $\Psi_k$ extends to an isometry with respect to the $L^2$-scalar products, induced by the volume forms vol$(g^M)$ and vol$(g^N)$. 
\begin{prop}\label{unitary} The separation of variables map \eqref{separation} extends to an isometric identification of $L^2-$Hilbert spaces
\begin{align*}
\Psi_k: L^2([0,1], L^2(\wedge^{k-1}T^*N\oplus \wedge^kT^*N, \textup{vol}(g^N)), dx)\to L^2(\wedge^kT^*M, \textup{vol}(g^M)).
\end{align*}
\end{prop}\ \\
\\[-7mm] Under this identification we obtain for the exterior derivative, as in [BS, (5.5)]
\begin{equation}\label{derivative} 
\Psi_{k+1}^{-1} d_k \Psi_k= \left( \begin{array}{cc}0&(-1)^k\partial_x\\0&0\end{array}\right)+\frac{1}{x}\left( \begin{array}{cc}d_{k-1,N}&c_k\\0&d_{k,N}\end{array}\right),
\end{equation}
where $c_k=(-1)^k(k-n/2)$ and $d_{k,N}$ denotes the exterior derivative on differential forms over $N$ of degree $k$. Taking adjoints we find 
\begin{equation}\label{coderivative} 
\Psi_k^{-1} d_k^t \Psi_{k+1}= \left( \begin{array}{cc}0&0\\(-1)^{k+1}\partial_x&0\end{array}\right)+\frac{1}{x}\left( \begin{array}{cc}d_{k-1,N}^t&0\\c_k&d_{k,N}^t\end{array}\right).
\end{equation}
\\[3mm] Consider now the Gauss-Bonnet operator $D_{GB}^+$ mapping forms of even degree to forms of odd degree. The Gauss-Bonnet operator acting on forms of odd degree is simply the formal adjoint $D_{GB}^-=(D_{GB}^+)^t$. With respect to $\Psi_{+}:=\oplus \Psi_{2k}$ and $\Psi_{-}:=\oplus \Psi_{2k+1}$ the relevant operators take the following form:
\begin{align}
&\Psi^{-1}_-D_{GB}^+\Psi_+ = \frac{d}{dx}+\frac{1}{x}S_0, \quad \Psi^{-1}_+D_{GB}^-\Psi_- = -\frac{d}{dx}+\frac{1}{x}S_0, \label{gauss-bonnet}\\
 &\Psi^{-1}_+\triangle^+ \Psi_+= \Psi^{-1}_+(D_{GB}^+)^t\Psi_-\Psi_-^{-1}D_{GB}^+ \Psi_+ = -\frac{d^2}{dx^2}+ \frac{1}{x^2}S_0(S_0+1),\label{laplacian} \\ \nonumber
  &\Psi^{-1}_-\triangle^- \Psi_-= \Psi^{-1}_-(D_{GB}^-)^t\Psi_+\Psi_+^{-1}D_{GB}^- \Psi_- = -\frac{d^2}{dx^2}+ \frac{1}{x^2}S_0(S_0-1).
\end{align}
where $S_0$ is a first order elliptic differential operator on $\Omega^*(N)$. The transformed Gauss-Bonnet operators in \eqref{gauss-bonnet} are regular singular in the sense of [BS] and [Br, Section 3]. Moreover, the Laplace Operator on $k$-forms over $M$ transforms to 
\begin{align}\label{A_k,S_0}
\Psi_k \triangle_k \Psi_k^{-1}=-\frac{d^2}{dx^2}+ \frac{1}{x^2}A_k,
\end{align}
where $A_k$ is a symmetric differential operator of order two, a restriction of $S_0(S_0+(-1)^k)$ to $\Omega^{k-1}(N)\oplus \Omega^k(N)$ to $\Omega^{k-1}(N)\oplus \Omega^k(N)$. The non-product situation on the bounded generalized cone $M$ is pushed into the $x$-dependence of the tangential part of the Laplacian. 
\\[3mm] We continue below under the isometric transformation $\Phi$. In particular the de Rham complex and the associated Laplace operators are considered in their equivalent form under the isometry.

\section{Decomposition of the de Rham complex}\
\\[-3mm] Following [L3], we decompose the de Rham complex of $M$ into a direct sum of subcomplexes of two types. The first type of the subcomplexes is given as follows. Let $\psi \in E_{\la}^k, \la \in V_k, k=1,..,n$ be a fixed non-zero generator of $E_{\la}^k$. Put
\begin{align*}
&\xi_1:=(0,\psi)\in \Omega^{k-1}(N)\oplus \Omega^{k}(N), \\ &\xi_2:=(\psi,0)\in \Omega^{k}(N)\oplus \Omega^{k+1}(N), \\
&\xi_3:=(0,\frac{1}{\sqrt{\lambda}}d_N\psi)\in \Omega^{k}(N)\oplus \Omega^{k+1}(N), \\ &\xi_4:=(\frac{1}{\sqrt{\lambda}}d_N\psi,0)\in \Omega^{k+1}(N)\oplus \Omega^{k+2}(N).
\end{align*}
Then $C^{\infty}_0((0,1),\langle \xi_1,\xi_2,\xi_3,\xi_4\rangle)$ is invariant under $d,d^t$ and we obtain a subcomplex of the de Rham complex:
\begin{align}\label{complex-1}
0 \rightarrow C_0^{\infty}((0,1),\left< \xi_1\right>) \xrightarrow{d^{\psi}_0} C_0^{\infty}((0,1),\left<\xi_2,\xi_3\right>) \xrightarrow{d^{\psi}_1}C_0^{\infty}((0,1),\left<\xi_4\right>)\rightarrow 0,
\end{align}
where $d^{\psi}_0,d^{\psi}_1$ take the following form with respect to the chosen basis:
\begin{align*}
d_0^{\psi}=\left(\begin{array}{c}(-1)^k\partial_x+\frac{c_k}{x}\\ x^{-1}\sqrt{\la}\end{array}\right), \quad d_1^{\psi}=\left(x^{-1}\sqrt{\la}, \ (-1)^{k+1}\partial_x+\frac{c_{k+1}}{x}\right).
\end{align*}
By [BV1, Theorem 4.10] separating out the subcomplex above is compatible with the boundary conditions of $\triangle_{k}^{rel}, \triangle_{k+1}^{rel}, \triangle_{k+2}^{rel}.$ In other words we have a decomposition into reducing subspaces of the Laplacians, see [W2] and [BV1, Subsection 4.2] for further details. Hence the relative boundary conditions induce self-adjoint extensions of the Laplacians of the subcomplex:
\begin{align}
\dom (\triangle_k^{rel})\cap L^2((0,1),E_{\la}^k)=\dom ((d_0^{\psi})_{\max}^td_{0,\min}^{\psi})=:\dom(\triangle_{0,\lambda}^k), \\
\dom (\triangle_{k+2}^{rel})\cap L^2((0,1),d_N E_{\la}^k)=\dom (d_{1,\min}^{\psi}(d_1^{\psi})_{\max}^t)=:\dom(\triangle_{2,\lambda}^k), \\
\dom (\triangle_{k+1}^{rel})\cap L^2((0,1),\widetilde{E}_{\la}^k)=\dom ((d_0^{\psi})_{\max}^td_{0,\min}^{\psi}+d_{1,\min}^{\psi}(d_1^{\psi})_{\max}^t)=:\dom(\triangle^k_{\lambda}).
\end{align}
Next we compute the associated Laplacians
\begin{align}\label{psi-laplacians}
\triangle_0^{\psi}:=(d_0^{\psi})^td_0^{\psi}=-\partial_x^2+\frac{1}{x^2}\left[\eta +\left(k+\frac{1}{2}-\frac{n}{2}\right)^2-\frac{1}{4}\right]=d_1^{\psi}(d_1^{\psi})^t=:
\triangle_2^{\psi}.
\end{align}
under the identification of any $\phi =f\cdot \xi_{i}\in C^{\infty}_0((0,1),\langle \xi_{i}\rangle),i=1,4$ with its scalar part $f \in C^{\infty}_0(0,1)$. We continue under this identification from here on.
\\[3mm] The second type of the subcomplexes comes from the harmonics on the base manifold $N$ and is given as follows. Consider $\mathcal{H}^k(N)$ and fix an orthonormal basis $\{u_i\},i=1,..,\dim \mathcal{H}^k(N)$ of $\mathcal{H}^k(N)$. Observe that for any $i$ the subspace $C^{\infty}_0((0,1),\langle0\oplus u_i,u_i\oplus 0\rangle)$ is invariant under $d,d^t$ and we obtain a subcomplex of the de Rham complex
\begin{align}\label{complex-2}
0\to C^{\infty}_0((0,1),\langle 0\, \oplus \, &u^k_i,\rangle)\xrightarrow{d} C^{\infty}_0((0,1),\langle u^k_i\oplus 0\rangle) \to 0, \\
&d=(-1)^k\partial_x+\frac{c_k}{x},
\end{align}
where the action of $d$ is identified with its scalar action. We continue under this identification. As for the subcomplex \eqref{complex-1}, separating out the subcomplex above is compatible with the relative boundary conditions by [BV1, Theorem 4.10]. Hence we obtain for the induced self-adjoint extensions
\begin{align}\nonumber
\dom (\triangle_k^{rel})\cap L^2((0,1),&\langle 0\oplus u^k_i\rangle)=\dom (d^t_{\max}d_{\min})=\\ \label{H1-rel}&=\dom\left((-1)^{k+1}\partial_x+\frac{c_k}{x}\right)_{\max}\left((-1)^{k}\partial_x+\frac{c_k}{x}\right)_{\min}, \\ \nonumber
\dom (\triangle_{k+1}^{rel})\cap L^2((0,1),&\langle u^k_i \oplus 0 \rangle)=\dom (d_{\min}d^t_{\max})=\\ \label{H2-rel}&=\dom\left((-1)^{k}\partial_x+\frac{c_k}{x}\right)_{\min}\left((-1)^{k+1}\partial_x+\frac{c_k}{x}\right)_{\max}.
\end{align}
By the Hodge decomposition on the base manifold $N$ the de Rham complex $(\Omega_0^*(M),d)$ decomposes completely into subcomplexes of the two types above. This decomposition gives in each degree $k$ a compatible decomposition for $\triangle_k^{rel}$, as observed [BV1, Theorem 4.10]. In the language of [W2] we have a decomposition into reducing subspaces of the Laplacians. Hence the Laplacians $\triangle_k^{rel}$ induce self-adjoint relative extensions of the Laplacians of the subcomplexes. In particular each subcomplex contributes to the function in \eqref{T-function} as follows.
\\[3mm] The relative boundary conditions turn the complex \eqref{complex-1} of the first type into a Hilbert complex (see [BL1]) of the following general form:
\begin{align*}
0 \rightarrow  H_k  \xrightarrow{D}  H_{k+1}  \xrightarrow{D}  H_{k+2} \rightarrow  0.
\end{align*}
By the specific form of the subcomplex we have the following relation between the zeta-functions corresponding to the Laplacians of the subcomplex
\begin{align}\label{spectral-relation1}
\zeta_{k+1}(s)=\zeta_{k}(s)+\zeta_{k+2}(s).
\end{align}
From the spectral relation \eqref{spectral-relation1} we deduce that the contribution of the subcomplex $H$ to the function $T(M,s)$ amounts to
\begin{align}
\nonumber \frac{(-1)^k}{2}\left[k\zeta_k(s)-(k+1)\zeta_{k+1}(s) +(k+2)\zeta_{k+2}(s)\right]\\ =\frac{(-1)^k}{2}(\zeta_{k+2}(s)-\zeta_{k}(s)).\label{contribution}
\end{align} 
Since there are in fact infinitely many subcomplexes of the first type, we first have to add up the contributions for $Re(s)$ large and then continue the sum analytically to $s=0$. Then the derivative at zero gives the contribution to $T(M)$. 
\\[3mm] For the contribution of the subcomplexes \eqref{complex-2} of the second kind to the analytic torsion, note that the relative boundary conditions turn the complex of second type into a Hilbert complex of the following general form:
\begin{align*}
0 \rightarrow  H_k  \xrightarrow{D}  H_{k+1}  \rightarrow  0.
\end{align*}
There are only finitely many such subcomplexes, since $\dim \mathcal{H}^*(N)<\infty$. Hence we obtain directly for the contribution to $\log T(M)$ from each of such subcomplexes
\begin{align}\label{spectral-relation2}
\frac{(-1)^{k+1}}{2}\zeta'(D^*D,s=0).
\end{align}

\section{Symmetry in the Decomposition}\label{symmetry-decomposition} \
\\[-3mm] In this section we present a symmetry of the de Rham complex on a model cone, as elaborated by M. Lesch in [L3]. Consider the subcomplexes \eqref{complex-1} of the first type
\begin{align*}
0 \rightarrow C_0^{\infty}((0,1),\left< \xi_1\right>) \xrightarrow{d_0} C_0^{\infty}((0,1),\left<\xi_2,\xi_3\right>) \xrightarrow{d_1}C_0^{\infty}((0,1),\left<\xi_4\right>)\rightarrow 0.
\end{align*}
with the associated Laplacians (identified with their scalar action)
\begin{align}
\triangle_0^{\psi}:=(d_0^{\psi})^td_0^{\psi}=-\partial_x^2+\frac{1}{x^2}\left[\eta +\left(k+\frac{1}{2}-\frac{n}{2}\right)^2-\frac{1}{4}\right]=d_1^{\psi}(d_1^{\psi})^t=:
\triangle_2^{\psi}.
\end{align}
By [BV1, Theorem 4.10] separating out the subcomplex above provides a decomposition into reducing subspaces of $\triangle_{k}^{rel}, \triangle_{k+1}^{rel}, \triangle_{k+2}^{rel}.$ Hence the relative boundary conditions induce self-adjoint extensions of the Laplacians $\triangle_{0}^{\psi}, \triangle_{2}^{\psi}$
\begin{align*}
\triangle_{0,rel}^{\psi}=(d_0^{\psi})_{\max}^td_{0,\min}^{\psi}, \quad \triangle_{2,rel}^{\psi}=d_{1,\min}^{\psi}(d_1^{\psi})_{\max}^t
\end{align*}
Now we discuss the relative boundary conditions for $\triangle_0^{\psi}, \triangle_2^{\psi}$. Assume that the lowest non-zero eigenvalue $\eta$ of $\triangle_{k,N}$ is $\eta > 1$. This can always be achieved by an appropriate scaling of the metric on $N$ 
\begin{align}\label{scaling}
g^{N,c}:=c^{-2}g^N, \quad c >0 \ \textup{large enough}.
\end{align}
More precisely, the Laplacian $\triangle_N^c$ defined on $\Omega^*(N)$ with respect to $g^{N,c}$ is related to the original Laplacian $\triangle_N$ as follows $$\triangle_N^c=c^2\triangle_N.$$ Hence indeed for $c>0$ sufficiently large we achieve that the Laplacian $\triangle_N^c$ has no "small" non-zero eigenvalues.
\\[3mm] This guarantees that $\triangle_0^{\psi}$ and $\triangle_2^{\psi}$ are in the limit point case at $x=0$ and hence all their self-adjoint extensions in $L^2(0,1)$ coincide at $x=0$. Hence we only need to consider the relative boundary conditions at $x=1$. With [BV1, Proposition 4.5], which is essentially the trace theorem of L. Paquet in [P], we obtain
\begin{align*}
&\mathcal{D}(\triangle_{0,rel}^{\psi})=\{f\in\mathcal{D}(\triangle_{0,\max}^{\psi})|f(1)=0\}, \\ &\mathcal{D}(\triangle_{2,rel}^{\psi})=\{f\in \mathcal{D}(\triangle_{2,\max}^{\psi})|(-1)^kf'(1)+c_{k+1}f(1)=0\}.
\end{align*}
The values $f(1),f'(1)$ are well-defined since $\mathcal{D}(\triangle_{0,2,\max}^{\psi})\!\subset H^2_{\textup{loc}}(0,1]$. 
\begin{remark}
The assumption on the lower bound of the non-zero eigenvalues of $\triangle_{k,N}$ can be dropped. Then the discussion of a finite direct sum of model Laplacians in the limit circle case enters the calculations. This setup has been elaborated in [BV1, Section 5].
\end{remark} \ \\
\\[-7mm] Next we consider the twin-subcomplex, associated to the subcomplex discussed above. Let $\phi:=*_N\psi \in \Omega^{n-k}(N)$. Put 
\begin{align*}
&\widetilde{\xi_1}:=(0,\frac{1}{\sqrt{\eta}}d^t_N\phi)\in \Omega^{n-k-2}(N)\oplus \Omega^{n-k-1}(N), \\ &\widetilde{\xi_2}:=(\frac{1}{\sqrt{\eta}}d^t_N\phi,0)\in \Omega^{n-k-1}(N)\oplus \Omega^{n-k}(N), \\
&\widetilde{\xi_3}:=(0,\phi)\in \Omega^{n-k-1}(N)\oplus \Omega^{n-k}(N), \\ &\widetilde{\xi_4}:=(\phi,0)\in \Omega^{n-k}(N)\oplus \Omega^{n-k+1}(N).
\end{align*}
Again the subspace $C_0^{\infty}((0,1),\langle \widetilde{\xi_1},\widetilde{\xi_2},\widetilde{\xi_3},\widetilde{\xi_4}\rangle)$ is invariant under the action of $d$ and $d^t$ and in fact we obtain a complex
\begin{align*}
0 \rightarrow C_0^{\infty}((0,1),\langle \widetilde{\xi_1}\rangle) \xrightarrow{d_0^{\phi}} C_0^{\infty}((0,1),\langle\widetilde{\xi_2},\widetilde{\xi_3}\rangle) \xrightarrow{d_1^{\phi}}C_0^{\infty}((0,1),\langle\widetilde{\xi_4}\rangle)\rightarrow 0.
\end{align*}
By computing explicitly the action of the exterior derivative \eqref{derivative} on the basis elements $\widetilde{\xi_i}$ we obtain 
\begin{align*}
d_0^{\phi}=\left(\begin{array}{c}(-1)^{n-k-1}\partial_x+\frac{c_{n-k-1}}{x}\\ x^{-1}\sqrt{\eta}\end{array}\right), \quad d_1^{\phi}=\left(x^{-1}\sqrt{\eta}, \ (-1)^{n-k}\partial_x+\frac{c_{n-k}}{x}\right).
\end{align*}
As for the first subcomplex we compute the relevant Laplacians:
\begin{align}\label{phi-laplacians}
\triangle_0^{\phi}=\triangle_2^{\phi}=-\partial_x^2+\frac{1}{x^2}\left[\eta +\left(k+\frac{1}{2}-\frac{n}{2}\right)^2-\frac{1}{4}\right]=\triangle_0^{\psi}=\triangle_2^{\psi},
\end{align}
where the operators are identified with their scalar actions. As before, separating out the subcomplex above, we decompose $\triangle^{rel}_{n-k\pm 1}, \triangle^{rel}_{n-k}$ compatibly. Hence the relative boundary conditions induce self-adjoint extensions
\begin{align*}
\triangle_{0,rel}^{\phi}=(d_0^{\phi})_{\max}^td_{0,\min}^{\phi}, \quad \triangle_{2,rel}^{\phi}=d_{1,\min}^{\phi}(d_1^{\phi})_{\max}^t
\end{align*}
of the Laplacians $\triangle_0^{\phi},\triangle_2^{\phi}$ respectively. Under the scaling assumption of \eqref{scaling} the relative boundary conditions for this pair of operators are computed to
\begin{align*}
&\mathcal{D}(\triangle_{0,rel}^{\phi})=\{f\in\mathcal{D}(\triangle_{0,\max}^{\phi})|f(1)=0\}, \\ &\mathcal{D}(\triangle_{2,rel}^{\phi})=\{f\in \mathcal{D}(\triangle_{2,\max}^{\phi})|(-1)^{n-k+1}f'(1)+c_{n-k}f(1)=0\},
\end{align*}
with [BV1, Proposition 4.5]. As before the values $f(1),f'(1)$ are well-defined since $\mathcal{D}(\triangle_{0,2,\max}^{\phi})\subset H^2_{\textup{loc}}(0,1]$. 
\\[3mm] So in total we obtain four self-adjoint operators, which differ only by their boundary conditions. Unfortunately the differences in the domains do not allow to cancel the contribution of the two twin-subcomplexes to the analytic torsion. However the symmetry still allows us to perform explicit computations. 
\\[3mm] Recall that $\psi$ was chosen to be a normalized coclosed $\eta$-eigenform on $N$ of degree $k$ and $\phi=*_N\psi$. Denote the dependence of the generating forms $\psi$ and $\phi$ on the eigenvalue $\eta$ by $\psi(\eta)$ and $\phi (\eta)$. Introduce further the notation
\begin{align*}
D(k):=&\{\lambda \in \textup{Spec}\triangle^{\psi(\eta)}_{0,rel} | \eta \in \textup{Spec} \triangle_{k, ccl, N}\backslash \{0\}\} \\
= &\{\lambda \in \textup{Spec}\triangle^{\phi(\eta)}_{0,rel} | \eta \in \textup{Spec} \triangle_{k, ccl, N}\backslash \{0\}\}, \\
&\hspace{10mm} N_1(k):=\{\lambda \in \textup{Spec}\triangle^{\psi(\eta)}_{2,rel} | \eta \in \textup{Spec} \triangle_{k, ccl, N}\backslash \{0\}\} , \\
&\hspace{10mm} N_2(k):=\{\lambda \in \textup{Spec}\triangle^{\phi(\eta)}_{2,rel} | \eta \in \textup{Spec} \triangle_{k, ccl, N}\backslash \{0\}\} ,
\end{align*}
where all eigenvalues are counted according to their multiplicities. Using this notation we now introduce the following zeta-functions for $Re(s)\gg0$ 
\begin{align*}
\zeta_D^k(s):=\sum\limits_{\lambda \in D(k)} \lambda ^{-s}, \ \zeta^k_{N_1}(s):=\sum\limits_{\lambda \in N_1(k)} \lambda^{-s}, \ \zeta^k_{N_2}(s):=\sum\limits_{\lambda \in N_2(k)}\lambda^{-s}, Re(s) \gg0.
\end{align*}
The $D$-subscript is aimed to point out that the zeta-functions in the sum are associated to Laplacians with Dirichlet boundary conditions at $x=1$. Similarly the $N$-subscripts point out the generalized Neumann boundary conditions at $x=1$, which are however different for $\triangle_2^{\phi}$ and $\triangle_2^{\psi}$. 
\\[3mm] The zeta-functions $\zeta_D^k(s), \zeta^k_{N_1}(s)$ and $\zeta^k_{N_2}(s)$ are by Theorem \ref{cone-zeta-function} holomorphic for $Re(s)$ sufficiently large, since they sum over eigenvalues of $\triangle^{rel}_*$ but with lower multiplicities. In view of \eqref{contribution}, which describes the contribution to analytic torsion from the subcomplexes, we set for $Re(s)$ large
\begin{defn}\label{zeta}
$\zeta_k(s):=\zeta^k_{N_1}(s)-\zeta^k_{D} (s))+(-1)^{n-1}(\zeta^k_{N_2}(s)-\zeta^k_{D}(s)).$
\end{defn}
\begin{remark}
Note that $\zeta_D^k(s)$ in the definition of $\zeta_k(s)$ cancel for $m=\dim M$ odd, simplifying the expression for $\zeta_k(s)$ considerably. Further simplifications (notably Proposition \ref{P}) take place throughout the discussion, so that an effective result can be obtained in the end.
\end{remark} \ \\
\\[-7mm]  Below we provide the analytic continuation of $\zeta_k(s)$ to $s=0$ for any fixed degree $k<\dim N-1$  and compute $(-1)^k \zeta'_k(0)$. The contribution coming from the subcomplexes of the second type \eqref{complex-2}, induced by the harmonic forms on the base $N$, is not included in $\zeta_k(s)$ and will be determined explicitly in a separate discussion.
\begin{remark}\label{double-count}
The total contribution of subcomplexes \eqref{complex-1} of first type to the logarithmic scalar analytic torsion $\log T(M)$ of the odd-dimensional bounded generalized cone $M$ is given by $$\frac{1}{2}\sum\limits_{k=0}^{n/2-1}(-1)^k\cdot \zeta'_k(0).$$ For an even-dimensional cone $M$ the zeta-function $\zeta_k(s)$ counts in the degree $k=(n-1)/2$ each subcomplex of type \eqref{complex-1} twice. Thus the total contribution of subcomplexes of first type to $\log T(M)$ is given by $$\frac{1}{2}\left(\sum\limits_{k=0}^{(n-3)/2}(-1)^k\cdot \zeta'_k(0)+\frac{(-1)^{\frac{(n-1)}{2}}}{2}\cdot \zeta'_{\frac{n-1}{2}}(0)\right),$$ where the first sum is set to zero for $\dim M=n+1=2$.
\end{remark}

\section{Some auxiliary analysis}\label{auxiliary}\
\\[-3mm] Fix a real number $\nu >1$ and consider the following differential operator 
$$l_{\nu}:=-\frac{d^2}{dx^2}+\frac{\nu^2-1/4}{x^2}:C^{\infty}_0(0,1)\to C^{\infty}_0(0,1).$$ 
By the choice $\nu >1$ the operator $l_{\nu}$ is in the limit point case at $x=0$ and hence its maximal and the minimal extensions coincide at $x=0$. Therefore we only need to fix boundary conditions at $x=1$ to define a self-adjoint extension of $l_{\nu}$. Put for $\A\in \R^*$:
$$\dom (L_{\nu}(\A)):=\{f \in \dom (l_{\nu,\max})|(\A-1/2)^{-1}f'(1)+f(1)=0\},$$
where $\A=\infty$ defines the Dirichlet boundary conditions at $x=1$ and $\A=1/2$ $-$ the pure Neumann boundary conditions at $x=1$.
\begin{prop}\label{bded-below}
The self-adjoint operator $L_{\nu}(\A), \A\in \R^*$ is discrete and bounded from below. For $\A^2 < \nu^2$ and $\A=\infty$ the operator $L_{\nu}(\A)$ is positive.
\end{prop}  
\begin{proof}
The discreteness of $L_{\nu}(\A)$ is asserted in [BS2], see also [L, Theorem 1.1] where this result is restated. For semi-boundedness of $L_{\nu}(\A)$ note that the potential $(\nu^2-1/4)/x^2$ is positive. Hence it suffices to discuss semi-boundedness of $-d^2/dx^2$ under different boundary conditions. By [W2, Theorem 8.24] all the self-adjoint extensions of $$-\frac{d^2}{dx^2}:C^{\infty}_0(0,1)\to C^{\infty}_0(0,1)$$ are bounded from below, since $-d^2/dx^2$ on $C^{\infty}_0(0,1)$ is semi-bounded. Indeed for any $f\in C^{\infty}_0(0,1)$ $$\langle -f'', f\rangle_{L^2(0,1)}=-\left. f'(x)\overline{f(x)}\right|_0^1+\int_0^1|f'(x)|^2dx\geq 0.$$ Hence $L_{\nu}(\A)$ is indeed semi-bounded. It remains to identify the lower bound in the case $\A^2< \nu^2$ and $\A=\infty$. For this consider any $f \in \dom (l_{\nu, \max}), f\not \equiv 0$. Recall $\dom (l_{\nu, \max}) \subset H^2_{loc}(0,1]$, which implies that $f$ is continuously differentiable at $(0,1)$ and $f,f'$ extend continuously to $x=1$. Moreover we infer from [BV1, Proposition 2.10 (iii)] the asymptotic behaviour $f(x)=O(x^{3/2}), f'(x)=O(x^{1/2})$, as $x\to 0$. We compute via integration by parts for any $\epsilon \in \R$ with $\epsilon^2 < \nu^2$:
\begin{align*}
\langle l_{\nu} f, f \rangle_{L^2(0,1)}=\int_0^1\left[ \left(-\frac{d}{dx}+\frac{\epsilon -1/2}{x}\right)\left(\frac{d}{dx}+\frac{\epsilon -1/2}{x}\right)f(x) + \right.\\ + \left. \frac{\nu^2-\epsilon^2}{x^2}f(x)\right]\cdot \overline{f(x)}dx = -\left.\left(f'(x)+\frac{\epsilon-1/2}{x}f(x)\right)\overline{f(x)}\, \right|_{x\to 0+}^1+\\+ \left\|f'+\frac{\epsilon-1/2}{x}f\right\|^2_{L^2(0,1)}+ \left\|\frac{\sqrt{\nu^2-\epsilon^2}}{x}f\right\|^2_{L^2(0,1)}.
\end{align*}
Now the asymptotics of $f(x)$ and $f'(x)$ as $x\to 0$ implies together with $f\not \equiv 0$:
\begin{align}\label{formel1}
\langle l_{\nu} f, f \rangle_{L^2(0,1)} > -\overline{f(1)}\cdot (f'(1)+(\epsilon-1/2)f(1)).
\end{align}
Evaluation of the conditions at $x=1$ for $L_{\nu}(\A)$ with $\A=\infty$ or $\A^2 < \nu^2$ proves the statement. 
\end{proof}
\begin{cor}\label{bessel-zeros}
Let $J_{\nu}(z)$ denote the Bessel function of first kind and put for any fixed $\A \in \R^*$ $$\widetilde{J}^{\A}_{\nu}(z):=\A J_{\nu}(z)+z J'_{\nu}(z),$$
where for $\A=\infty$ we put $\widetilde{J}^{\A}_{\nu}(z):=J_{\nu}(z)$. Then for $\nu >1$ and $\A=\infty$ or $\A^2 < \nu^2$, the zeros of $\widetilde{J}^{\A}_{\nu}(z)$ are real, discrete and symmetric about the origin. The eigenvalues of the positive operator $L_{\nu}(\A)$ are simple and given by squares of positive zeros of $\widetilde{J}^{\A}_{\nu}(z)$, i.e.
$$\textup{Spec}L_{\nu}(\A)=\{\mu^2 | \widetilde{J}^{\A}_{\nu}(\mu)=0, \mu >0\}.$$
\end{cor}
\begin{proof}
The general solution to $l_{\nu}f=\mu^2f, \mu\neq 0$ is of the following form
$$f(x)=c_1\sqrt{x}J_{\nu}(\mu x)+c_2\sqrt{x}Y_{\nu}(\mu x),$$
where $c_1,c_2$ are constants and $J_{\nu}, Y_{\nu}$ denote Bessel functions of first and second kind, respectively. For $\nu >1$ the asymptotic behaviour of $f\in \dom (l_{\nu, \max})$ is given by $f(x)=O(x^{3/2}), x\to 0$. Hence a solution to $l_{\nu}f=\mu^2f, \mu \neq 0$ with $f\in \dom (l_{\nu, \max})$ must be of the form $$f(x)=c_1\sqrt{x}J_{\nu}(\mu x).$$
Taking in account the boundary conditions for $L_{\nu}(\A)$ with at $\A=\infty$ or $\A^2 < \nu^2$, we deduce correspondence between zeros of $\widetilde{J}^{\A}_{\nu}(z)$ and eigenvalues of $L_{\nu}(\A)$. Hence by Proposition \ref{bded-below} we deduce the statements about the zeros of $\widetilde{J}^{\A}_{\nu}(z)$, up to the statement on the symmetry of zeros, which follows simply from the standard infinite series representation of Bessel functions. 
\\[3mm] Furthermore, $J_{\nu}(-\mu x)=(-1)^{\nu}J_{\nu}(\mu x), \mu \neq 0$ and hence each eigenvalue $\mu^2$ of $L_{\nu}(\A)$ is simple with the unique (up to a multiplicative constant) eigenfunction $f(x)=\sqrt{x}J_{\nu}(\mu x), \mu >0$ This completes the proof. 
\end{proof}\ \\
\\[-7mm] Similar statements can be deduced for more general values of $\A\in \R^*$, but are not relevant in the present discussion. Finally note as a direct application of Proposition \ref{bded-below} that the Laplacians $\triangle^{\psi}_{0,2,rel}$ and $\triangle^{\phi}_{0,2,rel}$, introduced in the previous section, are positive.
\begin{cor}\label{G-lower-bound}
For all degrees $k=0,..,\dim N$ we have 
\begin{align*}
D(k)\subset \R^+, \ N_i(k)\subset \R^+, \ i=1,2.
\end{align*}
\end{cor}\ \\
\\[-7mm] Next consider the zeta-function $\zeta (s, L_{\nu}(\A)), \A\in \R^*$ associated to the self-adjoint realization $L_{\nu}(\A)$ of $l_{\nu}$. It is well-known, see [L, Theorem 1.1] that the zeta-function extends meromorphically to $\C$ with the analytic representation given by the Mellin transform of the heat trace:
$$\zeta(s,L_{\nu}(\A))= \frac{1}{\Gamma(s)}\int_0^{\infty}t^{s-1}\textup{Tr}_{L^2}(e^{-tL_{\nu}(\A)}P)dt,$$ 
where $P$ is the projection on the orthogonal complement of the null space of $L_{\nu}(\A)$. The heat operator $\exp (-tL_{\nu}(\A))$ is defined by the spectral theorem and is a bounded smoothing operator with finite trace $\textup{Tr}_{L^2}(e^{-tL_{\nu}(\A)}P)$ of standard polylogarithmic asymptotics as $t \to 0+$, see [Ch, Theorem 2.1]. We can write for $t > 0$ $$\textup{Tr}(e^{-tL_{\nu}(\A)})=\frac{1}{2\pi i}\int_{\Lambda}e^{-\lambda t}\textup{Tr}(\lambda-L_{\nu}(\A))^{-1}d\lambda,$$ where the contour $\Lambda$ shall encircle all non-zero eigenvalues of the semi-bounded $L_{\nu}(\A), \A \in \R^*$ and be counter-clockwise oriented, in analogy to Figure \ref{contour-spreafico} below.
\\[3mm] Now, following [S], we obtain an integral representation for $\zeta(s,L_{\nu}(\A))$ in a computationally convenient form. Introduce a numbering $(\lambda_n)$ of the eigenvalues of $L_{\nu}(\A)$ and observe $$\textup{Tr}(\lambda-L_{\nu}(\A))^{-1}=\sum_n\frac{1}{\lambda -\lambda_n}=\sum_n \frac{d}{d\lambda}\log \left(1-\frac{\lambda}{\lambda_n}\right),$$ where we fix henceforth the branch of logarithm in $\C\backslash \R^+$ with $0\leq Im \log z < 2\pi$. We continue with this branch of logarithm throughout the section. Integrating now by parts first in $\lambda$, then in $t$ we obtain
\begin{align}\label{integral-spreafico}
\zeta(s,L_{\nu}(\A))=\frac{s^2}{\Gamma(s+1)}\int_0^{\infty}t^{s-1}\frac{1}{2\pi i}\int_{\Lambda}\frac{e^{-\lambda t}}{-\lambda}\left[-\sum_n \log \left(1-\frac{\lambda}{\lambda_n}\right)\right] d\lambda dt.
\end{align}

\section{Contribution from the Subcomplexes I} \
\\[-3mm] We continue in the setting and in the notation of Section \ref{symmetry-decomposition} and fix any degree $k\leq \dim N-1$. We define the following contour:
\begin{align}\label{l-c-2}
\Lambda_c:=\{\lambda \in \C| |\textup{arg}(\lambda -c)|=\pi /4\}
\end{align}
oriented counter-clockwise, with $c>0$ a fixed positive number, smaller than the lowest non-zero eigenvalue of $\triangle^{rel}_*$. The contour is visualized in the Figure \ref{contour-spreafico}: 
\begin{figure}[h]
\begin{center}
\includegraphics[width=0.5\textwidth]{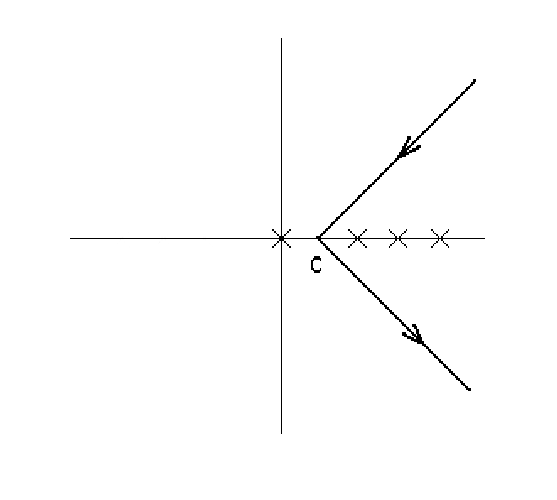}
\caption{The contour $\Lambda_c$. The $\times$'s represent the eigenvalues of $\triangle^{rel}_*$.}
\label{contour-spreafico}
\end{center}
\end{figure}
\\[-2mm] In analogy to the constructions of [S] we obtain for the zeta-functions $\zeta^k_D(s),\zeta^k_{N_1}(s),\zeta^k_{N_2}(s)$ the following results.
\begin{prop}\label{D-prop}
Let $M=(0,1]\times N, g^M=dx^2\oplus x^2g^N$ be a bounded generalized cone. Let the metric on the base manifold $N$ be scaled as in \eqref{scaling} such that the non-zero eigenvalues of the form-valued Laplacians on $N$ are bigger than $1$. Denote by $\triangle_{k,ccl,N}$ the Laplace Operator on coclosed k-forms on $N$. Let $$F_k:= \{\xi \in \R^+ \ | \ \xi^2=\eta +\left(k+1/2-n/2\right)^2, \eta \in \textup{Spec}\triangle_{k,ccl,N}\backslash\{0\} \}.$$ Then we obtain with $(j_{\nu,i})_{i\in \N}$ being the positive zeros of the Bessel function $J_{\nu}(z)$
\begin{align}
\zeta^k_D(s)=\frac{s^2}{\Gamma(s+1)}\int_0^{\infty}t^{s-1}\frac{1}{2\pi i}\int_{\wedge_c}\frac{e^{-\lambda t}}{-\lambda}T_D^k(s,\lambda)d\lambda dt,
\end{align}
\begin{align}
T_D^k(s,\lambda)=\sum_{\nu \in F_k}t_{\nu}^{D,k}(\lambda)\, \nu^{-2s}, \quad t_{\nu}^{D,k}(\lambda)=-\sum_{i=1}^{\infty}\log\left(1-\frac{{\nu}^2\lambda}{j_{\nu,i}^2}\right).
\end{align}
\end{prop}
\begin{proof}
Consider for $\eta \in \textup{Spec}\triangle_{k, ccl, N}\backslash \{0\}$ the operators $\triangle_0^{\psi(\eta)}$ and $\triangle_0^{\phi(\eta)}$, defined in \eqref{psi-laplacians} and \eqref{phi-laplacians}. Under the identification with their scalar parts we have
$$\triangle_0^{\phi(\eta)}= \triangle_0^{\psi(\eta)}=-\partial_x^2+\frac{1}{x^2}\left[\nu^2-\frac{1}{4}\right],$$
where $\nu:=\sqrt{\eta+(k+1/2-n/2)^2}$. By scaling of the metric on $N$ we have $\nu >1$ and hence the self-adjoint extensions $\triangle_{0,rel}^{\phi(\eta)}$ and $\triangle_{0,rel}^{\psi(\eta)}$ are determined only by their Dirichlet boundary conditions at $x=1$. By Corollary \ref{bessel-zeros} we obtain:
$$\zeta^k_D(s)=\sum_{\nu \in F_k}\sum_{i=1}^{\infty}j_{\nu ,i}^{-2s}=\sum_{\nu \in F_k}\nu^{-2s}\sum_{i=1}^{\infty}\left(\frac{j_{\nu,i}}{\nu}\right)^{-2s}, \quad Re(s)\gg0,$$ where $j_{\nu,i}$ are the positive zeros of $J_{\nu}(z)$. This series is well-defined for $Re(s)$ large by Theorem \ref{cone-zeta-function}, since $\triangle_{0,rel}^{\phi(\eta)}$($\equiv \triangle_{0,rel}^{\psi(\eta)}$) as direct sum components of $\triangle^{rel}_*$ have the same spectrum as $\triangle^{rel}_*$, but with lower multiplicities in general.
\\[3mm] Due to the uniform convergence of integrals and series we obtain with similar computations as for \eqref{integral-spreafico} an integral representation for this sum:
\begin{align}
\zeta_D^k(s)=\frac{s^2}{\Gamma(s+1)}\int_0^{\infty}t^{s-1}\frac{1}{2\pi i}\int_{\wedge_c}\frac{e^{-\lambda t}}{-\lambda}T_D^k(s,\lambda)d\lambda dt,
\end{align}
\begin{align}
T_D^k(s,\lambda)=\sum_{\nu \in F_k}t_{\nu}^{D,k}(\lambda)\, \nu^{-2s}, \quad t_{\nu}^{D,k}(\lambda)=-\sum_{i=1}^{\infty}\log\left(1-\frac{{\nu}^2\lambda}{j_{\nu,i}^2}\right).
\end{align}
Note that the contour $\Lambda_c$, defined in \eqref{l-c-2} encircles all eigenvalues of $\triangle_{0,rel}^{\phi(\eta)}\equiv \triangle_{0,rel}^{\psi(\eta)}$ by construction, since the operators are positive by Corollary \ref{G-lower-bound}.
\end{proof}
\begin{prop}\label{N-prop}
Let $M=(0,1]\times N, g^M=dx^2\oplus x^2g^N$ be a bounded generalized cone. Let the metric on the base manifold $N$ be scaled as in \eqref{scaling} such that the non-zero eigenvalues of the form-valued Laplacians on $N$ are bigger than $1$. Denote by $\triangle_{k,ccl,N}$ the Laplace Operator on coclosed k-forms on $N$. Let $$F_k:=\{\xi \in \R^+ \ | \ \xi^2=\eta +\left(k+1/2-n/2\right)^2, \eta \in \textup{Spec}\triangle_{k,ccl,N}\backslash\{0\} \}.$$ Then we obtain for $l=1,2$
\begin{align}
\zeta^k_{N_l}(s)=\frac{s^2}{\Gamma(s+1)}\int_0^{\infty}t^{s-1}\frac{1}{2\pi i}\int_{\wedge_c}\frac{e^{-\lambda t}}{-\lambda}T_{N_l}^k(s,\lambda)d\lambda dt,
\end{align}
\begin{align}
T_{N_l}^k(s,\lambda)=\sum_{\nu \in F_k}t_{\nu}^{N_l,k}(\lambda)\, \nu^{-2s}, \quad t_{\nu}^{N_l,k}(\lambda)=-\sum_{i=1}^{\infty}\log\left(1-\frac{{\nu}^2\lambda}{\widetilde{j}_{\nu,l,i}^2}\right),
\end{align}
where $(\widetilde{j}_{\nu,l,i})_{i\in \N}$ are the positive zeros of $\widetilde{J}_{\nu}^{N_l,k}(z)$ for $l=1,2$. The functions $\widetilde{J}_{\nu}^{N_l}(z)$ are defined as follows
\begin{align*}
\widetilde{J}_{\nu}^{N_1,k}(z):=\left(\frac{1}{2}+(-1)^kc_{k+1}\right)J_{\nu}(z)+z J_{\nu}'(z),  \\ \widetilde{J}_{\nu}^{N_2,k}(z):=\left(\frac{1}{2}+(-1)^kc_{k}\right)J_{\nu}(z)+z J_{\nu}'(z).
\end{align*}
\end{prop}
\begin{proof}
Consider for $\eta \in \textup{Spec}\triangle_{k, ccl, N}\backslash \{0\}$ the operators $\triangle_2^{\psi(\eta)}$ and $\triangle_2^{\phi(\eta)}$, defined in \eqref{psi-laplacians} and \eqref{phi-laplacians}, which contribute to the zeta-functions $\zeta^k_{N_1}(s)$ and $\zeta^k_{N_2}(s)$ correspondingly. Under the identification with their scalar parts we have
$$\triangle_2^{\phi(\eta)}= \triangle_2^{\psi(\eta)}=-\partial_x^2+\frac{1}{x^2}\left[\nu^2-\frac{1}{4}\right],$$
where $\nu:=\sqrt{\eta+(k+1/2-n/2)^2}$. By scaling of the metric on $N$ we have $\nu >1$ and hence the self-adjoint extensions $\triangle_{2,rel}^{\phi(\eta)}$ and $\triangle_{2,rel}^{\psi(\eta)}$ are determined only by their generalized Neumann boundary conditions at $x=1$. Recall
\begin{align*}
\mathcal{D}(\triangle_{2,rel}^{\psi})=\{f\in \mathcal{D}(\triangle_{2,\max}^{\psi})|f'(1)+(-1)^kc_{k+1}f(1)=0\}, \\
\mathcal{D}(\triangle_{2,rel}^{\phi})=\{f\in \mathcal{D}(\triangle_{2,\max}^{\phi})|f'(1)+(-1)^{n-k+1}c_{n-k}f(1)=0\}.
\end{align*}
Observe $(-1)^{n-k+1}c_{n-k}=(-1)^kc_k$ and put
\begin{align*}
\widetilde{J}_{\nu}^{N_1,k}(\mu):=\left(\frac{1}{2}+(-1)^kc_{k+1}\right)J_{\nu}(\mu)+\mu J_{\nu}'(\mu), \\ \widetilde{J}_{\nu}^{N_2,k}(\mu):=\left(\frac{1}{2}+(-1)^kc_{k}\right)J_{\nu}(\mu)+\mu J_{\nu}'(\mu).
\end{align*}
Note for any degree $k$ and any $\nu \in F_k$ $$\left|\frac{1}{2}+(-1)^kc_{k+1}\right|= \left|\frac{1}{2}+(-1)^kc_{k}\right|=\left|\frac{n}{2}-\frac{1}{2}-k\right|<\nu.$$
Hence by Corollary \ref{bessel-zeros} we obtain for $l=1,2$:
$$\zeta^k_{N_l}(s)=\sum_{\nu \in F_k}\sum_{i=1}^{\infty}\widetilde{j}_{\nu ,l,i}^{-2s}=\sum_{\nu \in F_k}\nu^{-2s}\sum_{i=1}^{\infty}\left(\frac{\widetilde{j}_{\nu,l,i}}{\nu}\right)^{-2s}, \quad Re(s)\gg0,$$ where $\widetilde{j}_{\nu,l,i}$ are the positive zeros of $\widetilde{J}_{\nu}^{N_l,k}(z)$ for $l=1,2$. This series is well-defined for $Re(s)$ large by Theorem \ref{cone-zeta-function}, since $\triangle_{2,rel}^{\phi(\eta)}, \triangle_{2,rel}^{\psi(\eta)}$ as direct sum components of $\triangle^{rel}_*$ have the same spectrum as $\triangle^{rel}_*$, but with lower multiplicities in general. 
\\[3mm] Due to the uniform convergence of integrals and series we obtain with similar computations as for \eqref{integral-spreafico} an integral representation for this sum:
\begin{align}
\zeta^k_{N_l}(s)=\frac{s^2}{\Gamma(s+1)}\int_0^{\infty}t^{s-1}\frac{1}{2\pi i}\int_{\wedge_c}\frac{e^{-\lambda t}}{-\lambda}T_{N_l}^k(s,\lambda)d\lambda dt,
\end{align}
\begin{align}
T^k_{N_l}(s,\lambda)=\sum_{\nu \in F_k}t_{\nu}^{N_l,k}(\lambda)\, \nu^{-2s}, \quad t_{\nu}^{N_l,k}(\lambda)=-\sum_{i=1}^{\infty}\log\left(1-\frac{{\nu}^2\lambda}{\widetilde{j}_{\nu,l,i}^2}\right).
\end{align}
Note that the contour $\Lambda_c$ encircles all the possible eigenvalues of $\triangle_{2,rel}^{\phi(\eta)}, \triangle_{2,rel}^{\psi(\eta)}$ by construction, since the operators are positive by Corollary \ref{G-lower-bound}.
\end{proof}  
\begin{cor}\label{contribution3}
Let $M=(0,1]\times N, g^M=dx^2\oplus x^2g^N$ be a bounded generalized cone. Let the metric on $N$ be scaled as in \eqref{scaling} such that the non-zero eigenvalues of the form-valued Laplacians on $N$ are bigger than $1$. Then we obtain with Definition \ref{zeta} in the notation of Propositions \ref{D-prop} and \ref{N-prop}
\begin{align}\label{integral-representation}
\zeta_k(s)= \frac{s^2}{\Gamma(s+1)}\int_0^{\infty}t^{s-1}\frac{1}{2\pi i}\int_{\wedge_c}\frac{e^{-\lambda t}}{-\lambda}T^k(s,\lambda)d\lambda dt,
\end{align}
$$
T^k(s,\lambda):=\!\sum_{\nu \in F_k}t^k_{\nu}(\lambda)\, \nu^{-2s},$$ $$t^k_{\nu}(\lambda):=\left[t_{\nu}^{N_1,k}(\lambda)-t_{\nu}^{D,k}(\lambda))+(-1)^{n-1}(t_{\nu}^{N_2,k}(\lambda)-t_{\nu}^{D,k}(\lambda))\right].
$$
If $\dim M$ is odd we obtain with $z:=\sqrt{-\lambda}$ and $\A_k:=n/2-1/2-k$
\begin{align*}
t^k_{\nu}(\lambda)=&\left[ -\log (\A_k I_{\nu}(\nu z) +\nu z I'_{\nu}(\nu z))+\log \left( 1+ \frac{ \A_k}{\nu}\right) + \right.\\ & \left. + \log (- \A_k I_{\nu}(\nu z) +\nu z I'_{\nu}(\nu z))-\log \left( 1- \frac{ \A_k}{\nu}\right)\right].
\end{align*}
For $\dim M$ even we compute with $z:=\sqrt{-\lambda}$
\begin{align*}
t^k_{\nu}(\lambda)=& \left[-\log (\A_k I_{\nu}(\nu z) +\nu z I'_{\nu}(\nu z))+\log \left( 1+ \frac{ \A_k}{\nu}\right) - \right.\\ & - \log (- \A_k I_{\nu}(\nu z) +\nu z I'_{\nu}(\nu z))+\log \left( 1- \frac{ \A_k}{\nu}\right)+ \\ & \left. + 2\log (I_{\nu}(\nu z))+2\log \nu \right].
\end{align*}
\end{cor} 
\begin{proof} Recall for convenience the definition of $\zeta_k(s)$ in Definition \ref{zeta} $$\zeta_k(s):=\zeta^k_{N_1}(s)-\zeta^k_{D} (s))+(-1)^{n-1}(\zeta^k_{N_2}(s)-\zeta^k_{D}(s)).$$ The integral representation and the definition of $t_{\nu}^k(\lambda)$ are then a direct consequence of Propositions \ref{D-prop} and \ref{N-prop}. It remains to present $t_{\nu}^k(\lambda)$ in terms of special functions.
\\[3mm] In order to simplify notation we put (recall $c_j:=(-1)^j(j-n/2)$)
\begin{align*}
\A_k:=\frac{1}{2}+(-1)^kc_{k+1}=\frac{n}{2}-\frac{1}{2}-k=-\left(\frac{1}{2}+(-1)^kc_{k}\right).\end{align*}
Now we present $t^{D,k}_{\nu}(\lambda)$ and $t^{N_l,k}_{\nu}(\lambda),l=1,2$ in terms of special functions.
This can be done by referring to tables of Bessel functions in [GRA] or [AS]. However in the context of the paper it is more appropriate to derive the presentation from results on zeta-regularized determinants. Here we follow the approach of [L, Section 4.2] in a slightly different setting. 
\\[3mm] The original setting of [L, (4.22)] provides an infinite product representation for $I_{\nu}(z)$. We apply its approach in order to derive the corresponding result for $\widetilde{I}_{\nu}^N(z):=\A I_{\nu}(z)+zI'_{\nu}(z)$, with $\A \in \{\pm \A_k\}$ and $\nu \in F_k$. 
\\[3mm] Consider now the following regular-singular Sturm-Liouville operator and its self-adjoint extension with $\A \in \{\pm \A_k\}$ and $\nu \in F_k$
$$l_{\nu}:=-\frac{d^2}{dx^2}+\frac{1}{x^2}\left(\nu^2-\frac{1}{4}\right):C^{\infty}_0(0,1)\to C^{\infty}_0(0,1),$$ $$ \mathcal{D}(L_{\nu}(\A)):=\{f\in\mathcal{D}(l_{\nu, \max})|f'(1)+(\A-1/2)f(1)=0\}.
$$ \ \\
Note we have $\A^2<\nu^2$ by construction and in particular $\A \neq -\nu$. Thus we find by Proposition \ref{bded-below} that $\ker L_{\nu}(\A)=\{0\}$ and
\begin{align}\label{hilfe-referenz}
\det\nolimits_{\zeta}(L_{\nu}(\A))=\sqrt{2\pi}\frac{\A+\nu}{2^{\nu}\Gamma(\nu+1)}.
\end{align}
Denote by $\phi(x,z),\psi(x,z)$ the solutions of $(l_{\nu}+z^2)f=0$, normalized in the sense of [L, (1.38a), (1.38b)] at $x=0$ and $x=1$, respectively. The general solution to $(l_{\nu}+z^2)f=0$ is of the following form $$f(x)=c_1\sqrt{x}I_{\nu}(zx)+c_2\sqrt{x}K_{\nu}(zx).$$ Applying the normalizing conditions of [L, (1.38a), (1.38b)] we obtain straightforwardly
\begin{align*}
&\psi(1,z)=1, \quad \psi'(1,z)=1/2 -\A, \\ &\phi(1,z)=2^{\nu}\Gamma(\nu+1)z^{-\nu}I_{\nu}(z) \ \textnormal{with} \ \phi(1,0)=1, \\ 
&\phi'(1,z)=2^{\nu}\Gamma(\nu+1)z^{-\nu}(I_{\nu}(z)\cdot 1/2+zI'_{\nu}(z))\ \textnormal{with} \ \phi'(1,0)=\nu +1/2.
\end{align*} 
Finally by [L, Proposition 4.6] we obtain with $\{\lambda_n\}_{n\in \N}$ being a counting of the eigenvalues of $L_{\nu}(\A)$:
\begin{align}\label{product}
\det\nolimits_{\zeta}(L_{\nu}(\A)+z^2)=\det\nolimits_{\zeta}(L_{\nu}(\A))\cdot \prod_{n=1}^{\infty}\left(1+\frac{z^2}{\lambda_n}\right).
\end{align}
Since $\ker L_{\nu}(\A)=\{0\}$, for all $n\in \N$ we have $\lambda_n\neq 0$.
Denote the positive zeros of $\widetilde{J}_{\nu}^N(z):=\A J_{\nu}(z)+zJ'_{\nu}(z)$ by $(\widetilde{j}_{\nu,i})_{i\in \N}$. Note in the notation of Proposition \ref{N-prop} that for $\A=\A_k$, $\widetilde{J}_{\nu}^N(z)=\widetilde{J}_{\nu}^{N_1,k}(z)$ and for $\A=-\A_k$, $\widetilde{J}_{\nu}^N(z)=\widetilde{J}_{\nu}^{N_2,k}(z)$. Observe by Corollary \ref{bessel-zeros}:
$$\textnormal{Spec}(L_{\nu}(\A))=\{\widetilde{j}_{\nu,i}^2|i\in \N\}.$$ Using the product formula \eqref{product} and [L, Theorem 1.2] applied to $L_{\nu}(\A)+z^2$, we compute in view of \eqref{hilfe-referenz}
\begin{align}\nonumber
\prod_{i=1}^{\infty}\left(1+\frac{z^2}{\widetilde{j}^2_{\nu,i}}\right)=\frac{W(\phi(\cdot, z);\psi(\cdot , z))}{\A+\nu}=\frac{2^{\nu}\Gamma(\nu)}{z^{\nu}(1+\A/\nu)}(\A I_{\nu}(z)+zI'_{\nu}(z)) \\ \label{prod-Bessel} \Rightarrow \quad \widetilde{I}_{\nu}^N(z)\equiv \A I_{\nu}(z)+zI'_{\nu}(z)=\frac{z^{\nu}}{2^{\nu}\Gamma(\nu)}\left(1+\frac{\A}{\nu}\right)\prod_{i=1}^{\infty}\left(1+\frac{z^2}{\widetilde{j}_{\nu,i}^2}\right).
\end{align}
The original computations of [L, (4.22)] provide an analogous result for $I_{\nu}(z)$ $$I_{\nu}(z)=\frac{z^{\nu}}{2^{\nu}\Gamma(\nu +1)}\prod_{i=1}^{\infty}\left(1+\frac{z^2}{j_{\nu,i}^2}\right),$$ where $j_{\nu,i}$ are the positive zeros of $J_{\nu}(z)$. Finally in view of the series representations for $t^{D,k}_{\nu}(\lambda)$ and $t^{N_l,k}_{\nu}(\lambda),l=1,2$ derived in Propositions \ref{D-prop} and \ref{N-prop} we obtain with $z=\sqrt{-\lambda}$ 
\begin{align}\label{t-D-special}
t^{D,k}_{\nu}(\lambda)=-\log I_{\nu}(\nu z) +\log \left( \frac{(\nu z)^{\nu}}{2^{\nu}\Gamma (\nu+1)}\right), \\
\label{t-N-special}t^{N_l,k}_{\nu}(\lambda)=-\log (\A_l I_{\nu}(\nu z) +\nu z I'_{\nu}(\nu z))+\log \left( \frac{(\nu z)^{\nu}}{2^{\nu}\Gamma (\nu)}\left(1 +\frac{ \A_l}{\nu}\right)\right),
\end{align}
where $\A_l=\A_k$ if $l=1$ and $\A_l=-\A_k$ if $l=2$. Putting together these two results we obtain with Definition \ref{zeta} the statement of the corollary.
\end{proof} \ \\
\\[-7mm] Now we turn to the discussion of $T^k(s,\lambda)$. For this we introduce the following zeta-function for $Re(s)$ large: $$\zeta_{k,N}(s):=\sum_{\nu \in F_k}\nu^{-s}=\sum_{\nu \in F_k}(\nu^2)^{-s/2},$$ where $\nu \in F_k$ are counted with their multiplicities and the second equality is clear, since $\nu\in F_k$ are positive. Recall that $\nu \in F_k$ solves $$\nu^2=\eta +(k+1/2-n/2)^2, \ \eta \in \textup{Spec}\triangle_{k,ccl,N}\backslash \{0\}$$ and hence $\zeta_{k,N}(2s)$ is simply the zeta-function of $\triangle_{k,ccl,N}+(k+1/2-n/2)^2$. By standard theory $\zeta(2s)$ extends (note that $\zeta(2s)$ can be presented by an alternating sum of zeta functions of $\triangle_{j,N}+(k+1/2-n/2)^2, j=0,..,k$) to a meromorphic function with possible simple poles at the usual locations $\{(n-p)/2 | p \in \N\}$ and $s=0$ being a regular point. Thus the $1/\nu^r$ dependence in $t^k_{\nu}(\lambda)$ causes a non-analytic behaviour of $T^k(s,\lambda)$ at $s=0$ for $r=1,..,n$, since $$\sum_{\nu \in F_k}\nu^{-2s}\frac{1}{\nu^r}=\zeta_{k,N}(2s+r)$$ possesses possibly a pole at $s=0$. Therefore the first $n=\dim N$ leading terms in the asymptotic expansion of $t^k_{\nu}(\lambda)$ for large orders $\nu$ are to be removed. We put
\begin{align}\label{p}
t^k_{\nu}(\lambda)=:p^k_{\nu}(\lambda)+\sum_{r=1}^n\frac{1}{\nu^r}f^k_r(\lambda), \quad P^k(s,\lambda):=\sum_{\nu >1}p^k_{\nu}(\lambda)\, \nu^{-2s}.
\end{align}
In order to get explicit expressions for $f^k_r(\lambda)$ we need following expansions of Bessel-functions for large order $\nu$, see [O, Section 9]:
\begin{align*}
I_{\nu}(\nu z) \sim \frac{1}{\sqrt{2\pi \nu}}\frac{e^{\nu \eta}}{(1+z^2)^{1/4}} \left[1+\sum_{r=1}^{\infty}\frac{u_r(t)}{\nu^r} \right], \\
I'_{\nu}(\nu z) \sim \frac{1}{\sqrt{2\pi \nu}} \frac{e^{\nu \eta}}{z (1+z^2)^{-1/4}}\left[1+\sum_{r=1}^{\infty}\frac{v_r(t)}{\nu^r} \right],
\end{align*} 
where we put $z:=\sqrt{-\lambda}, t:=(1+z^2)^{-1/2}$ and $\eta:=1/t+\log (z/(1+1/t))$. Recall that $\lambda \in \Lambda_c$, defined in \eqref{l-c-2}. The induced $z=\sqrt{-\lambda}$ is contained in $\{z\in \C||\textup{arg}(z)|<\pi /2\}\cup \{ix|x\in (-1,1)\}$. This is precisely the region of validity for these asymptotic expansions, determined in [O, (7.18)].
\\[3mm] Same expansions are quoted in [BKD, Section 3]. In particular we have as in [BKD, (3.15)] the following expansion in terms of orders
\begin{align}\label{polynom1}
&\log \left[ 1+\sum_{r=1}^{\infty}\frac{u_r(t)}{\nu^r} \right] \sim \sum_{r=1}^{\infty}\frac{D_r(t)}{\nu^r}, \\ \label{polynom2}
&\log \left[ \left(1+\sum_{k=1}^{\infty}\frac{v_r(t)}{\nu^r} \right)\pm \frac{\A_k}{\nu}t\left(1+\sum_{r=1}^{\infty}\frac{u_r(t)}{\nu^r}\right)\right] \sim \sum_{r=1}^{\infty}\frac{M_r(t, \pm \A_k)}{\nu^r},
\end{align}
where $D_r(t)$ and $M_r(t,\pm \A_k)$ are polynomial in $t$. Using these series representations we prove the following result.
\begin{lemma}\label{f-r}
For $\dim M$ being odd we have with $z:=\sqrt{-\lambda}$, $t:=(1+z^2)^{-1/2}=1/\sqrt{1-\lambda}$ and $\A_k=n/2-1/2-k$
$$f^k_r(\lambda)=M_r(t,-\A_k)-M_r(t,+\A_k)+(-1)^{r+1}\frac{\A_k^r-(-\A_k)^r}{r}.$$ For $\dim M$ being even we have in the same notation $$f^k_r(\lambda)=-M_r(t,-\A_k)-M_r(t,+\A_k)+2D_r(t) +(-1)^{r+1}\frac{\A_k^r+(-\A_k)^r}{r}.$$ 
\end{lemma}
\begin{proof}
We get by the series representation \eqref{polynom1} and \eqref{polynom2} the following expansions for large orders $\nu$: 
\begin{align*}
\log (\pm \A_k I_{\nu}(\nu z) +\nu z I'_{\nu}(\nu z)) &\sim \log \left( \frac{\nu}{\sqrt{2\pi \nu}} \frac{e^{\nu \eta}}{z (1+z^2)^{-1/4}} \right)+\sum_{r=1}^{\infty}\frac{M_r(t, \pm \A_k)}{\nu^r}, \\
\log (I_{\nu}(\nu z)) &\sim \log \left( \frac{1}{\sqrt{2\pi \nu}} \frac{e^{\nu \eta}}{ (1+z^2)^{1/4}} \right)+\sum_{r=1}^{\infty}\frac{D_r(t)}{\nu^r}.
\end{align*}
Furthermore, with $\nu > |\A_k|$ for $\nu \in F_k$ we obtain 
$$\log (1\pm \frac{\A_k}{\nu})=\sum_{r=1}^{\infty}(-1)^{r+1}\frac{(\pm \A_k)^r}{r\nu^r}.$$ Hence in total we obtain an expansion for $t_{\nu}^k(\lambda)$ in terms of orders $\nu$:
\begin{align*}
t_{\nu}^k(\lambda)\sim \sum_{r=1}^{\infty}\frac{1}{\nu^r}\left(M_r(t,-\A_k)-M_r(t,+\A_k)+(-1)^{r+1}\frac{\A_k^r-(-\A_k)^r}{r}\right), \\ \textup{for $\dim M$ odd}, \\
t_{\nu}^k(\lambda)\sim \sum_{r=1}^{\infty}\frac{1}{\nu^r}\!\left(\! 2D_r(t)-M_r(t,-\A_k)-M_r(t,+\A_k)+(-1)^{r+1}\frac{\A_k^r+(-\A_k)^r}{r}\right) \\ + \log \left( \frac{\lambda}{\lambda -1}\right), \quad \textup{for $\dim M$ even}. 
\end{align*}
From here the explicit result for $f_r^k(\lambda)$ follows by its definition.
\end{proof} \ 
\\[-1mm] From the integral representation \eqref{integral-representation} we find that the singular behaviour enters the zeta-function in form of 
\begin{align*}
\sum_{r=1}^n\frac{s^2}{\Gamma(s+1)}\zeta_{k,N}(2s+r)\int_0^{\infty}t^{s-1}\frac{1}{2\pi i}\int_{\wedge_c}\frac{e^{-\lambda t}}{-\lambda}f^k_r(\lambda)d\lambda dt.
\end{align*}
We compute explicitly this contribution coming from $f^k_r(\lambda)$ in terms of the polynomial structure of $M_r$ and $D_r$. It can be derived from \eqref{polynom1} and \eqref{polynom2}, see also [BKD, (3.7), (3.16)], that the polynomial structure of $M_r$ and $D_r$ is given by
\begin{align*}
D_r(t)=\sum_{b=0}^{r}x_{r,b}t^{r+2b},\quad M_r(t, \pm \A_k)=\sum_{b=0}^{r}z_{r,b}(\pm \A_k)t^{r+2b}.
\end{align*}
\begin{lemma} \label{ff} 
For $\dim M$ odd we obtain
\begin{align*}
\int_0^{\infty}t^{s-1}\frac{1}{2\pi i}\int_{\wedge_c}\frac{e^{-\lambda t}}{-\lambda}f^k_r(\lambda)d\lambda dt =\\= \sum_{b=0}^{r}(z_{r,b}(-\A_k)-z_{r,b}(\A_k))\frac{\Gamma(s+b+r/2)}{s\Gamma (b+r/2)}.
\end{align*}
For $\dim M$ even we obtain
\begin{align*}
\int_0^{\infty}t^{s-1}\frac{1}{2\pi i}\int_{\wedge_c}\frac{e^{-\lambda t}}{-\lambda}f^k_r(\lambda)d\lambda dt =\\= \sum_{b=0}^{r}(2x_{r,b}-z_{r,b}(-\A_k)-z_{r,b}(\A_k))\frac{\Gamma(s+b+r/2)}{s\Gamma (b+r/2)}.
\end{align*}
\end{lemma}
\begin{proof} Observe from [GRA, 8.353.3] by substituting the new variable $x=\lambda-1$, with $a>0$:
\begin{align*}
\frac{1}{2\pi i}\int_{\wedge_c}\frac{e^{-\lambda t}}{-\lambda}\frac{1}{(1-\lambda)^a}d\lambda= \frac{1}{2\pi i}e^{-t}\!\!\int_{\wedge_{c-1}}-\frac{e^{-x t}}{x+1}\frac{1}{(-x)^{a}}dx =\\=\frac{1}{\pi}\sin(\pi a)\Gamma(1-a)\Gamma(a,t).
\end{align*}
Using now the relation between the incomplete Gamma function and the probability integral
$$\int_0^{\infty}t^{s-1}\Gamma(a,t)dt=\frac{\Gamma(s+a)}{s}$$ we obtain 
\begin{align*}
\int_0^{\infty}t^{s-1}\frac{1}{2\pi i}\int_{\wedge_c}\frac{e^{-\lambda t}}{-\lambda}\frac{1}{(1-\lambda)^a}d\lambda dt = \\ \frac{1}{\pi}\sin \left(\pi a\right)\Gamma (1-a)\frac{\Gamma(s+a)}{s}= \frac{\Gamma (s+a)}{s\Gamma (a)}.
\end{align*}
Further note for $t>0$ $$\frac{1}{2\pi i}\int_{\wedge_c}\frac{e^{-\lambda t}}{-\lambda}d\lambda=0,$$
since the contour $\Lambda_c$ does not encircle the pole $\lambda=0$ of the integrand. Hence the $\lambda -$independent part of $f_r^k(\lambda)$ vanishes after integration. The statement is now a direct consequence of Lemma \ref{f-r}. 
\end{proof}
\ \\
\\[-8mm] Next we derive asymptotics of $p^k_{\nu}(\lambda):=t^k_{\nu}(\lambda)-\sum_{r=1}^n\frac{1}{\nu^r}f^k_r(\lambda)$ for large arguments $\lambda$ and fixed order $\nu$
\begin{prop}\label{AB} For large arguments $\lambda$ and fixed order $\nu$ we have the following asymptotics
\begin{align*}
p^k_{\nu}(\lambda)=a^k_{\nu}\log (-\lambda)+b^k_{\nu}+O\left((-\lambda)^{-1/2}\right),
\end{align*}
where for $\dim M$ odd 
$$a^k_{\nu}=0, \quad b^k_{\nu}=\left((\log \left(1+\frac{\A_k}{\nu}\right)-\log \left(1-\frac{\A_k}{\nu}\right)- \sum_{r=1}^n(-1)^{r+1}\frac{\A_k^r-(-\A_k)^r}{r\nu^r}\right),$$ and for $\dim M$ even
$$a^k_{\nu}=-1, \quad b^k_{\nu}=\left(\log \left(1+\frac{\A_k}{\nu}\right)+\log \left(1-\frac{\A_k}{\nu}\right) - \sum_{r=1}^n(-1)^{r+1}\frac{\A_k^r+(-\A_k)^r}{r\nu^r}\right).$$
\end{prop}
\begin{proof}
For large argument $\lambda$ we obtain $$t=\frac{1}{\sqrt{1+z^2}}=\frac{1}{\sqrt{1-\lambda}}=O\left((-\lambda)^{-1/2}\right).$$ Therefore the polynomials $M_r(t,\pm \A_k)$ and $D_r(t)$, having no constant terms, are of asymptotics $O\left((-\lambda)^{-1/2}\right)$ for large $\lambda$. Hence directly from Lemma \ref{f-r} we obtain in odd dimensions for large $\lambda$
\begin{align}\label{f-r-asymptotic-odd}
\frac{f^k_r(\lambda)}{\nu^r}\sim (-1)^{r+1}\frac{(\A_k)^r-(-\A_k)^r}{r\nu^r}+O\left((-\lambda)^{-1/2}\right).
\end{align}
In even dimensions we get 
\begin{align}\label{f-r-asymptotic-even}
\frac{f^k_r(\lambda)}{\nu^r}\sim (-1)^{r+1}\frac{(\A_k)^r+(-\A_k)^r}{r\nu^r}+O\left((-\lambda)^{-1/2}\right).
\end{align}
It remains to identify explicitly the asymptotics of $t_{\nu}^k(\lambda)$. Note by [AS, p. 377] the following expansions for large arguments and fixed order: 
$$I_{\nu}(z)=\frac{e^z}{\sqrt{2\pi z}}\left(1+O\left(\frac{1}{z}\right)\right), \quad I'_{\nu}(z)=\frac{e^z}{\sqrt{2\pi z}}\left(1+O\left(\frac{1}{z}\right)\right).$$ These expansions hold for $|\textup{arg}(z)|<\pi /2$ and in particular for $z=\sqrt{-\lambda}$ with $\lambda \in \Lambda_c$ large, where $\Lambda_c$ is defined in \eqref{l-c-2}. Further observe for such $z=\sqrt{-\lambda}, \lambda \in \Lambda_c$ large:
\begin{align*}
\log \left(1+O\left(\frac{1}{z}\right)\right)=O\left((-\lambda)^{-1/2}\right), \\ \Rightarrow \ \log (\pm \A_k +\nu z)=\log z + \log \nu + \log \left(1\pm \frac{\A_k}{\nu z}\right)= \\ =\log z + \log \nu + O\left((-\lambda)^{-1/2}\right).
\end{align*}
Together with the expansions of the Bessel-functions we obtain for $t_{\nu}^k(\lambda)$ defined in Corollary \ref{contribution3}
\begin{align*}
t_{\nu}^k(\lambda)=\log \left( 1+\frac{\A_k}{\nu}\right)-\log \left( 1-\frac{\A_k}{\nu}\right) + O\left((-\lambda)^{-1/2}\right), \\ \textup{for $\dim M$ odd}, \\
t_{\nu}^k(\lambda)=-\log (-\lambda)+\log \left( 1+\frac{\A_k}{\nu}\right)+\log \left( 1-\frac{\A_k}{\nu}\right) + O\left((-\lambda)^{-1/2}\right), \\ \textup{for $\dim M$ even}.
\end{align*}
Recall the definition of $p_{\nu}^k(\lambda)$ in \eqref{p}. Combining this with \eqref{f-r-asymptotic-odd} and \eqref{f-r-asymptotic-even} we obtain the desired result.
\end{proof}
\begin{defn} \label{AB1} With the coefficients $a_{\nu}^k$ and $b_{\nu}^k$ defined in Proposition \ref{AB}, we set for $Re(s)\gg0$ 
$$A^k(s):=\sum_{\nu \in F_k}a^k_{\nu}\, \nu^{-2s}, \quad B^k(s):=\sum_{\nu \in F_k}b^k_{\nu}\, \nu^{-2s}.$$
\end{defn} \ \\
Now the last step towards the evaluation of the zeta-function of Corollary \ref{contribution3} is the discussion of $$P^k(s,\lambda):=\sum_{\nu \in F_k}p^k_{\nu}(\lambda)\, \nu^{-2s}, \ Re(s)\gg0.$$ At this point the advantage of taking in account the symmetry of the de Rham complex is particularly visible:
\begin{prop}\label{P}
\begin{align*}
P^k(s,0)=0.
\end{align*}
\end{prop}
\begin{proof}
As $\lambda \to 0$ we find that $t=(1-\lambda)^{-1/2}$ tends to $1$. Since as in [BGKE, (4.24)]
\begin{align}\label{DM}
M_r(1,\pm \A_k)=D_r(1)+(-1)^{r+1}\frac{(\pm\A_k)^r}{r}
\end{align}
we find with Lemma \ref{f-r} that in both the even- and odd-dimensional case $f^k_r(\lambda) \to 0$ as $\lambda \to 0$. Thus we simply need to study the behaviour of $t^k_{\nu}(\lambda)$ defined in Corollary \ref{contribution3} for small arguments. The results follow from the asymptotic behaviour of Bessel functions of second order for small arguments which holds without further restrictions on $z$
$$I_{\nu}(z)\sim \frac{1}{\Gamma (\nu+1)}\left( \frac{z}{2} \right)^{\nu}, \ |z|\to 0.$$
Using the relation $I'_{\nu}(z)=\frac{1}{2}(I_{\nu+1}(z)+I_{\nu-1}(z)$ we compute as $|z|\to 0$
\begin{align*}
\pm \A_k I_{\nu}(\nu z)+\nu z I'_{\nu}(\nu z)&\sim \frac{\nu}{\Gamma (\nu +1)}\left(\frac{\nu z}{2}\right)^{\nu}\left[1\pm \frac{\A_k}{\nu}+\frac{\nu z^2}{4(\nu+1)}\right], \\
\nu I_{\nu}(\nu z)&\sim \frac{\nu}{\Gamma (\nu +1)}\left(\frac{\nu z}{2}\right)^{\nu}.
\end{align*}
The result now follows from the explicit form of $t_{\nu}^k(\lambda)$.
\end{proof}
\begin{remark}
The statement of Proposition \ref{P} shows an obvious advantage of taking in account the symmetry of the de Rham complex.
\end{remark}
\ 
\\[-3mm] Now we have all the ingredients together, since by analogous arguments as in [S, Section 4.1] the total zeta-function of Corollary \ref{contribution3} is given as follows:
\begin{align*}
\zeta_k(s) &=  \frac{s}{\Gamma (s+1)}[\gamma A^k(s)-B^k(s)-\frac{1}{s}A^k(s)+P^k(s,0)]\ + \\ + \  \sum_{r=1}^n&\frac{s^2}{\Gamma (s+1)}\zeta_{k,N}(2s+r)\int_0^{\infty}t^{s-1}\frac{1}{2\pi i}\int_{\wedge_c}\frac{e^{-\lambda t}}{-\lambda}f^k_r(\lambda)d\lambda dt + \frac{s^2}{\Gamma (s+1)}h(s),
\end{align*}
where the last term vanishes with its derivative at $s=0$. Simply by inserting the results of Lemma \ref{ff}, Proposition \ref{AB}, Proposition \ref{P} together with Definition \ref{AB1} into the above expression we obtain the following proposition:
\begin{prop}\label{zeta-total} Continue in the setting of Corollary \ref{contribution3}. Up to a term of the form $s^2h(s) /\Gamma (s+1)$, which vanishes with its derivative at $s=0$, the zeta-function $\zeta_k(s)$ from Definition \ref{zeta} is given in odd dimensions by
\begin{align*}
\frac{s}{\Gamma (s+1)}\left[ \sum_{\nu \in F_k}\nu^{-2s}\log\left(1-\frac{\A_k}{\nu}\right)-\sum_{\nu \in F_k}\nu^{-2s}\log\left(1+\frac{\A_k}{\nu}\right)\right.+ \\ + \sum_{r=1}^n\zeta_{k,N}(2s+r)\left.(-1)^{r+1}\frac{\A_k^r-(-\A_k)^r}{r}\right]+\\ +\sum_{r=1}^n\zeta_{k,N}(2s+r)\frac{s}{\Gamma (s+1)}\left[\sum_{b=0}^{r}(z_{r,b}(-\A_k)-z_{r,b}(\A_k))\frac{\Gamma(s+b+r/2)}{\Gamma (b+r/2)}\right].
\end{align*}
In even dimensions we obtain
\begin{align*}
\frac{s}{\Gamma (s+1)}\left[ -\sum_{\nu \in F_k}\nu^{-2s}\log\left(1-\frac{\A_k}{\nu}\right)-\sum_{\nu \in F_k}\nu^{-2s}\log\left(1+\frac{\A_k}{\nu}\right) +\right.\\\left.  +\sum_{\nu \in F_k}\nu^{-2s}\left(\frac{1}{s}-\gamma \right)+ \sum_{r=1}^n\zeta_{k,N}(2s+r)(-1)^{r+1}\frac{\A_k^r+(-\A_k)^r}{r}\right]+\\ +\sum_{r=1}^n\zeta_{k,N}(2s+r)\frac{s}{\Gamma (s+1)}\left[\sum_{b=0}^{r}(2x_{r,b}-z_{r,b}(-\A_k)-\right.\\\left.-z_{r,b}(\A_k))\frac{\Gamma(s+b+r/2)}{\Gamma (b+r/2)}\right].
\end{align*}
\end{prop}
\begin{cor}\label{total-contribution-1}
With $\zeta_{k,N}(s,a):=\sum_{\nu\in F_k} (\nu+a)^{-s}$ we deduce for odd dimensions
\begin{align*}
\zeta_k'(0)=& \,\,\zeta_{k,N}'(0,\A_k)-\zeta_{k,N}'(0,-\A_k)+\\+&\sum_{i=1}^n(-1)^{i+1}\frac{\A_k^i-(-\A_k)^i}{i}\textup{Res}\zeta_{k,N}(i)\left\{\frac{\gamma}{2}+\frac{\Gamma'(i)}{\Gamma(i)}\right\}+ \\+&
\sum_{i=1}^n\frac{1}{2}\,\textup{Res}\zeta_{k,N}(i)\sum_{b=0}^{i}\left(z_{i,b}(-\A_k)-z_{i,b}(\A_k)\right)\frac{\Gamma'(b+i/2)}{\Gamma (b+i/2)}.
\end{align*}
and for even dimensions
\begin{align*}
\zeta_k'(0)=& \,\,\zeta_{k,N}'(0,\A_k)+\zeta_{k,N}'(0,-\A_k)+\\+&\sum_{i=1}^n(-1)^{i+1}\frac{\A_k^i+(-\A_k)^i}{i}\textup{Res}\zeta_{k,N}(i)\left\{\frac{\gamma}{2}+\frac{\Gamma'(i)}{\Gamma(i)}\right\}+ \\+&
\sum_{i=1}^n\frac{1}{2}\,\textup{Res}\zeta_{k,N}(i)\sum_{b=0}^{i}\left(2x_{i,b}-z_{i,b}(-\A_k)-z_{i,b}(\A_k)\right)\frac{\Gamma'(b+i/2)}{\Gamma (b+i/2)}.
\end{align*}
\end{cor}
\begin{proof}
First we consider a major building brick of the expressions in Proposition \ref{zeta-total}. Here we follow the approach of [BKD, Section 11]. Put for $\A\in\{\pm \A_k\}$
\begin{align*}
K(s):=\sum_{\nu \in F_k}\nu^{-2s}\left[-\log\left(1+\frac{\A}{\nu}\right)+\sum_{r=1}^n(-1)^{r+1}\frac{1}{r}\left(\frac{\A}{\nu}\right)^r\right].
\end{align*}
Since the zeta-function $\zeta_{k,N}(s)=\sum_{\nu \in F_k}\nu^{-s}$ converges absolutely for $Re(s)\geq n+1, n=\dim N$, the sum above converges for $s=0$. In order to evaluate $K(0)$, introduce a regularization parameter $z$ as follows:
\begin{align*}
K_0(z):= \sum_{\nu \in F_k}\int_0^{\infty}t^{z-1}e^{-\nu t}\left(e^{-\A t}+\sum_{i=0}^n(-1)^{i+1}\frac{\A^it^i}{i!}\right) dt\\
=\Gamma(z)\cdot \zeta_{k,N}(z,\A)+\sum_{i=0}^n(-1)^{i+1}\frac{\A^i}{i!}\Gamma (z+i)\zeta_{k,N}(z+i), 
\end{align*}
where we have introduced
\begin{align*}
\zeta_{k,N}(z,\A):=\frac{1}{\Gamma(z)}\sum_{\nu \in F_k}\int_0^{\infty}t^{z-1}e^{-(\nu +\A)t}\, dt.
\end{align*}
For $Re(s)$ large enough $\zeta_{k,N}(z,\A)=\sum_{\nu \in F_k}(\nu +\A)^{-z},$
is holomorphic and extends meromorphically to $\C$, since it is the zeta-function of $\triangle_{k,ccl,N}+\A$. Note that for $\A\in \{\pm \A_k\}$ and $\nu \in F_k$ we have $\A\neq -\nu$, so no zero mode appears in the zeta function $\zeta_{k,N}(z,\A)$. In particular $K_0(z)$ is meromorphic in $z\in \C$ and by construction $$K_0(0)=K(0).$$ With the same arguments as in [BKD, Section 11] we arrive at
\begin{align*}
K(0)=&\zeta_{k,N}'(0,\A)
-\zeta_{k,N}'(0)+ \\ + &\sum_{i=1}^n(-1)^{i+1}\frac{\A^i}{i}\left[\textup{Res}\zeta_{k,N}(i)\left\{\!\gamma +\frac{\Gamma'(i)}{\Gamma(i)}\!\right\}+\textup{PP}\zeta_{k,N}(i)\right],\end{align*}
where PP$\zeta_{k,N}(r)$ denotes the constant term in the asymptotics of $\zeta_{k,N}(s)$ near the pole singularity $s=r$. This result corresponds to the result obtained in [BKD, p.388], where the factors $1/2$ in front of $\zeta'_{k,N}(0)$ and $2$ in front of $\textup{Res}\,\zeta_{k,N}(i)$, as present in [BKD], do not appear here because of a different notation: here we have set $\zeta_{k,N}(s)=\sum \nu^{-s}$ instead of $\sum \nu^{-2s}$. 
\\[3mm] In fact $K(0)$ enters the calculations twice: with $\A=\A_k$ and $\A=-\A_k$. In the odd-dimensional case both expressions are subtracted from each other, in the even-dimensional case they are added up. Furthermore we compute straightforwardly
\begin{align*}
\left.\frac{d}{ds}\right|_{s=0}\zeta_{k,N}(2s+r)\frac{s}{\Gamma (s+1)}\frac{\Gamma(s+b+r/2)}{\Gamma (b+r/2)}=\\=\frac{1}{2}\,\textup{Res}\zeta_{k,N}(r)\left[\frac{\Gamma'(b+r/2)}{\Gamma (b+r/2)}+\gamma \right]+\textup{PP}\zeta_{k,N}(r).
\end{align*}
We infer from \eqref{DM} 
\begin{align*}
\sum_{b=0}^{r}(z_{r,b}(-\A_k)-z_{r,b}(\A_k))=(-1)^r\frac{\A_k^r-(-\A_k)^r}{r}, \\
\sum_{b=0}^{i}\left(2x_{i,b}-z_{i,b}(-\A_k)-z_{i,b}(\A_k)\right)=(-1)^r\frac{\A_k^r+(-\A_k)^r}{r}.
\end{align*}
This leads after several cancellations to the desired result in odd dimensions. In even dimensions the result follows by a straightforward evaluation of the derivative at zero for the remaining component:
\begin{align*}
\left.\frac{d}{ds}\right|_{s=0}\frac{s}{\Gamma (s+1)}\zeta_{k,N}(2s)\left(\frac{1}{s}-\gamma\right)=2\zeta'_{k,N}(0).
\end{align*}
\end{proof}

\section{Contribution from the Subcomplexes II}\
\\[-3mm] It remains to identify the contribution to the analytic torsion coming from the subcomplexes \eqref{complex-2} of second type, induced by the harmonics on the base manifold $N$. The necessary calculations are provided in [L3] and are repeated here for completeness. Recall the explicit form of these subcomplexes 
\begin{align}
0\to C^{\infty}_0((0,1),\langle 0\, \oplus \, u_i\rangle)\xrightarrow{d} C^{\infty}_0((0,1),\langle u_i\oplus 0\rangle) \to 0,
\end{align}
where $\{u_i\}$ is an orthonormal basis of $\dim \mathcal{H}^k(N)$. With respect to the generators $0\oplus u_i$ and $u_i\oplus 0$ we obtain for the action of the exterior derivative $$d=(-1)^k\partial_x+\frac{c_k}{x}, \quad c_k=(-1)^k(k-n/2).$$
By compatibility of the induced decomposition we have (cf. \eqref{H1-rel})
\begin{align*}
\dom (\triangle_k^{rel})\cap L^2((0,R),&\langle 0\oplus u_i\rangle)=\dom (d^t_{\max}d_{\min})=\\ &=\dom\left((-1)^{k+1}\partial_x+\frac{c_k}{x}\right)_{\max}\left((-1)^{k}\partial_x+\frac{c_k}{x}\right)_{\min}.
\end{align*}
Consider, in the notation of Section \ref{auxiliary}, for any $\nu \in \R$ and $\A\in \R\cup \{\infty\}$ the operator $l_{\nu}=-\partial_x^2+x^{-2}(\nu^2-1/4)$ with the following self-adjoint extension:
\begin{align*}
\mathcal{D}(L_{\nu}(\A))=\{f\in \mathcal{D}(l_{\nu, \max})|(\A-1/2)^{-1}f'(1)+f(1)=0, \\ f(x)=O(\sqrt{x}), x\to 0\}.
\end{align*} 
Here $L_{\nu}(\A=1/2)$ denotes the self-adjoint extension of $l_{\nu}$ with pure Neumann boundary conditions at $x=1$. Furthermore $L_{\nu}(\infty)$ is the extension with Dirichlet boundary conditions at $x=1$. As a consequence of [BV1, Proposition 2.7] we have
$$\left((-1)^{k+1}\partial_x+\frac{c_k}{x}\right)_{\max}\left((-1)^{k}\partial_x+\frac{c_k}{x}\right)_{\min}=L_{|k-(n-1)/2|}(\infty).$$
It is well-known, see also [L, Theorem 1.1] and [L, (1.37)], that the zeta-function of $L_{\nu}(\A)$ extends meromorphically to $\C$ and is regular at the origin. We abbreviate $$T(L_{\nu}(\A)):=\log \det L_{\nu}(\A)=-\zeta'(s=0, L_{\nu}(\A)).$$
Put $b_k:=\dim \mathcal{H}^{k}(N)$. Then the contribution to the analytic torsion coming from harmonics on the base manifold is given due to the formula \eqref{spectral-relation2} as follows: 
\begin{align}\label{harmonics}
\frac{1}{2}\sum_{k=0}^{\dim M} (-1)^{k} b_k \,T(L_{|k-(n-1)/2|}(\infty))
\end{align}
\begin{prop}\label{spec} \ \\
For $\nu \geq 0$ we have $\textup{Spec}(L_{\nu}(\infty))\cup\{0\}=\textup{Spec}(L_{\nu+1}(\nu +1))\cup\{0\}$. 
\end{prop}
\begin{proof}
Put $d_p:=\partial_x+x^{-1}p$. We get $$l_{p+1/2}=d_p^td_p, \quad l_{p-1/2}=d_pd_p^t.$$
By a combination of [BV1, Proposition 2.5, 2.7], which determine the maximal and the minimal domains of $d_p$, we obtain for $\nu \geq 0$
\begin{align*}
\dom (d_{\nu+1/2, \max}d^t_{\nu+1/2,\min})=\{f\in \dom (l_{\nu,\max})|f(x)=O(\sqrt{x}), x\to 0, f(1)=0\}, \\
\dom (d^t_{\nu+1/2, \min}d_{\nu+1/2,\max})=\{f\in \dom (l_{\nu+1,\max})|f(x)=O(\sqrt{x}), x\to 0, \\ f'(1)+(\nu +1/2)f(1)=0\}.
\end{align*}
Hence we find
\begin{align*}
L_{\nu}(\infty)=d_{\nu+1/2, \max}d^t_{\nu+1/2,\min}=d_{\nu+1/2, \max}(d_{\nu+1/2,\max})^*, \\
L_{\nu+1}(\nu+1)=d^t_{\nu+1/2, \min}d_{\nu+1/2,\max}=(d_{\nu+1/2, \max})^*d_{\nu+1/2,\max}.
\end{align*}
Comparing both operators we deduce the statement on the spectrum, since all non-zero eigenvalues of the operators are simple by similar arguments as in Corollary \ref{bessel-zeros}. 
\end{proof}

\begin{prop}\label{lesch-determinant}
Let $\A+\nu\neq 0$. Then
\begin{align*}
T(L_{\nu}(\infty))=T(L_{\nu}(\A))-\log (\A+\nu).
\end{align*}
\end{prop}
\begin{proof} 
The assumption $\A+\nu\neq 0$ implies with [BV1, Corollary 3.11], which is a consequence of [L, Theorem 1.2]
$$\det\nolimits_{\zeta}(L_{\nu}(\A))=\sqrt{2\pi}\frac{\A+\nu}{2^{\nu}\Gamma(1+\nu)}.$$
Moreover we have by [BV1, Corollary 3.12], which in the present setup is a consequence of [L, Theorem 1.2]
$$\det\nolimits_{\zeta}(L_{\nu}(\infty))=\frac{\sqrt{2\pi}}{\Gamma(1+\nu)2^{\nu}}.$$
Consequently we obtain for $\A+\nu\neq 0$
$$\frac{\det\nolimits_{\zeta}(L_{\nu}(\infty))}{\det\nolimits_{\zeta}(L_{\nu}(\A))}=\frac{1}{\A+\nu}.$$
Taking logarithms we get the result.
\end{proof}

\begin{prop}\label{combinatorics}
\begin{align*}
T(L_{k+1/2}(\infty))=\log 2-\sum_{l=0}^k\log (2l+1).
\end{align*}
\end{prop}
\begin{proof}
Apply Proposition \ref{lesch-determinant} to $L_{\nu+1}(\nu+1), \nu\geq 0$. We obtain 
\begin{align*}
T(L_{\nu+1}(\infty))=T(L_{\nu+1}(\nu+1))-\log (2\nu+2)=
T(L_{\nu}(\infty))-\log (2\nu+2),
\end{align*}
where for the second equality we used Proposition \ref{spec}. We iterate the equality with $\nu =k-1/2$ and obtain
\begin{align*}
T(L_{k+1/2}(\infty))=T(L_{1/2}(\infty))-\sum_{l=0}^k\log (2l+1).
\end{align*}
The operator $L_{1/2}(\infty)$ is simply $-\partial_x^2$ on [0,1] with Dirichlet boundary conditions. Its spectrum is given by $(n^2\pi^2)_{n\in \N}$. Thus we obtain with $\zeta_R(0)=-1/2$ and $\zeta'_R(0)=-1/2\log 2\pi$
\begin{align*}
\zeta_{L_{1/2}(\infty)}(s)=\sum_{n=1}^{\infty}\pi^{-2s}n^{-2s} \\ \Rightarrow \ \zeta'_{L_{1/2}(\infty)}(0)=-2(\log \pi )\zeta_R(0)+2\zeta'_R(0)=-\log 2.
\end{align*}
\end{proof}\ \\
\\[-7mm] Now we finally compute the contribution from harmonics on the base:
\begin{thm}\label{harm-odd-even} Let $M$ be a bounded generalized cone of length one over a closed oriented Riemannian manifold $N$ of dimension $n$. Let $\chi(N)$ denote the Euler characteristic of $N$ and $b_k:=\dim \mathcal{H}^k(N)$ be the Betti numbers. 
\\[3mm] Then the contribution to the analytic torsion coming from harmonics on the base manifold is given as follows. For $\dim M$ odd the contribution amounts to
\begin{align*}
\frac{\log 2}{2}\chi(N)-\sum_{k=0}^{n/2-1}(-1)^kb_k \sum_{l=0}^{n/2-k-1}\log (2l+1)-\\ - \frac{1}{2}\sum_{k=0}^{n/2-1}(-1)^kb_k\log (n-2k+1).
\end{align*}
For $\dim M$ even the contribution amounts to
\begin{align*}
\frac{1}{2}\sum_{k=0}^{(n-1)/2}(-1)^kb_k\log (n-2k+1).
\end{align*}
\end{thm}
\begin{proof}
We infer from \eqref{harmonics} for the contribution of the harmonics on the base manifold
$$\frac{1}{2}\sum_{k=0}^{\dim M} (-1)^{k} b_k \,T(L_{|k-(n-1)/2|}(\infty)).$$ We obtain by Poincare duality on the base manifold $N$ 
\begin{align*}
&\textup{For $\dim M=n+1$ odd: }\frac{1}{2}\sum_{k=0}^{\dim M} (-1)^{k} b_k \,T(L_{|k-(n-1)/2|}(\infty))= \\ &\hspace{20mm}=
\frac{1}{2}\sum_{k=0}^{n/2-1} (-1)^{k} b_k \,(T(L_{n/2-k-1/2}(\infty))+T(L_{n/2-k+1/2}(\infty))), \\
&\textup{For $\dim M=n+1$ even: }\frac{1}{2}\sum_{k=0}^{\dim M} (-1)^{k} b_k \,T(L_{|k-(n-1)/2|}(\infty))= \\ &\hspace{20mm}=
\frac{1}{2}\sum_{k=0}^{(n-1)/2} (-1)^{k} b_k \,(T(L_{n/2-k-1/2}(\infty))-T(L_{n/2-k+1/2}(\infty))).
\end{align*}
Inserting the result of Proposition \ref{combinatorics} into the expressions above, we obtain the statement.
\end{proof}

\section{Total Result and Formulas in lower Dimensions}\label{an-torsion-general}\
\\[-3mm] Patching together the results of the both preceeding sections we can now provide a complete formula for the analytic torsion of a bounded generalized cone. In fact we simply have to add up the results of Theorem \ref{harm-odd-even} and Corollary \ref{total-contribution-1}. In even dimensions one has to be careful in the middle degree, as explained in Remark \ref{double-count}.
\begin{thm}\label{final-odd}
Let $M=(0,1]\times N, g^M=dx^2\oplus x^2g^N$ be an odd-dimensional bounded generalized cone over a closed oriented Riemannian manifold $(N,g^N)$. Let the metric on the base manifold $N$ be scaled such that the non-zero eigenvalues of the form-valued Laplacians on $N$ are bigger than one. Introduce the notation $n=\dim N, \ \A_k=(n-1)/2-k$ and $b_k=\dim \mathcal{H}^k(N)$. Put 
\begin{align*}
F_k:=\{\xi \in \R^+ \ | \ \xi^2=\eta +\left(k+1/2-n/2\right)^2, \eta \in \textup{Spec}\triangle_{k,ccl,N}\backslash\{0\} \}, \\
\zeta_{k,N}(s)=\sum_{\nu \in F_k} \nu^{-s},\quad \zeta_{k,N}(s, \A):=\sum_{\nu\in F_k} (\nu+\A)^{-s}, \quad Re(s)\gg0.
\end{align*}
Then the logarithm of the scalar analytic torsion of $M$ is given by
\begin{align*}
\log T(M)= \frac{\log 2}{2}\chi(N)-\sum_{k=0}^{n/2-1}(-1)^kb_k\sum_{l=0}^{n/2-k-1}\log (2l+1)-\\-\frac{1}{2}\sum_{k=0}^{n/2-1}(-1)^kb_k\log (n-2k+1)+\sum_{k=0}^{n/2-1}\frac{(-1)^k}{2} (\zeta_{k,N}'(0,\A_k)-\zeta_{k,N}'(0,-\A_k))+\\+ \sum_{k=0}^{n/2-1}\frac{(-1)^k}{2} \sum_{i=1}^n(-1)^{i+1}\frac{\A_k^i-(-\A_k)^i}{i}\textup{Res}\zeta_{k,N}(i)\left\{\frac{\gamma}{2}+\frac{\Gamma'(i)}{\Gamma(i)}\right\}+ \\+ \sum_{k=0}^{n/2-1}\frac{(-1)^k}{2}
\sum_{i=1}^n\frac{1}{2}\,\textup{Res}\zeta_{k,N}(i)\sum_{b=0}^{i}\left(z_{i,b}(-\A_k)-z_{i,b}(\A_k)\right)\frac{\Gamma'(b+i/2)}{\Gamma (b+i/2)}.
\end{align*} 
\end{thm}
\begin{thm}\label{final-even}
Let $M=(0,1]\times N, g^M=dx^2\oplus x^2g^N$ be an even-dimensional bounded generalized cone over a closed oriented Riemannian manifold $(N,g^N)$. Let the metric on the base manifold $N$ be scaled such that the non-zero eigenvalues of the form-valued Laplacians on $N$ are bigger than one. Introduce the notation $n=\dim N, \ \A_k=(n-1)/2-k$ and $b_k=\dim \mathcal{H}^k(N)$. Put 
\begin{align*}
F_k:=\{\xi \in \R^+ \ | \ \xi^2=&\eta +\left(k+1/2-n/2\right)^2, \eta \in \textup{Spec}\triangle_{k,ccl,N}\backslash\{0\} \}, \\
\zeta_{k,N}(s)&:=\sum_{\nu \in F_k} \nu^{-s},\quad \zeta_{k,N}(s, \A):=\sum_{\nu\in F_k} (\nu+\A)^{-s}, \quad Re(s)\gg0.
\\ \delta_k&:=\left\{ \begin{array}{cl} 1/2 & \textup{if} \ k=(n-1)/2, \\ 1 & \textup{otherwise}.\end{array}\right.
\end{align*}
Then the logarithm of the scalar analytic torsion of $M$ is given by
\begin{align*}
\log T(M)= \!\!\sum_{k=0}^{(n-1)/2}\!\frac{(-1)^k}{2}\left[b_k\log (n-2k+1)+\delta_k
\zeta_{k,N}'(0,\A_k)+\delta_k\zeta_{k,N}'(0,-\A_k)\right]\\+\sum_{k=0}^{(n-1)/2}\frac{(-1)^k}{2} \delta_k \sum_{i=1}^n(-1)^{i+1}\frac{\A_k^i+(-\A_k)^i}{i}\textup{Res}\zeta_{k,N}(i)\left\{\frac{\gamma}{2}+\frac{\Gamma'(i)}{\Gamma(i)}\right\}+ \\+ \sum_{k=0}^{(n-1)/2}\frac{(-1)^k}{2} \delta_k\sum_{i=1}^n\frac{1}{2}\,\textup{Res}\zeta_{k,N}(i)\sum_{b=0}^{i}\left(2x_{i,b}-z_{i,b}(-\A_k)-\right.\\ \left.-z_{i,b}(\A_k)\right)\frac{\Gamma'(b+i/2)}{\Gamma (b+i/2)}.
\end{align*} 
\end{thm} \ \\
The formula could not be made further explicit due to presence of coefficients $x_{r,b}$ and $z_{r,b}(\pm \A_k)$, arising from the polynomials
\begin{align*}
D_r(t)=\sum_{b=0}^{r}x_{r,b}t^{r+2b},\quad M_r(t, \pm \A)=\sum_{b=0}^{r}z_{r,b}(\pm \A)t^{r+2b},
\end{align*}
which were introduced in the expansions \eqref{polynom1} and \eqref{polynom2}. These polynomials can be computed explicitly for any given order $r\in \N$. To point out the applicability of the general results we pursue explicit computations in dimension two and three. We continue in the notation of the theorems above.
\begin{cor}\label{two-dim}
Let $M$ be a two-dimensional bounded generalized cone of length one over a closed oriented manifold $N$ with a metric scaled as in Theorem \ref{final-even}. Then the analytic torsion of $M$ is given by 
$$\log T(M)=\frac{1}{2}\dim H^0(N) \log 2 + \frac{1}{2}\zeta'_{0,N}(0)-\frac{1}{4}\textup{Res}\zeta_{0,N}(s=1).$$ In the special case of $N=S^1$ we obtain
$$\log T(M)=\frac{1}{2}\left(-\log \pi -1\right).$$
\end{cor}
\begin{proof}
In the two-dimensional case the general formula of Theorem \ref{final-even} reduces to the following expression:
\begin{align*}
\log T(M)=\frac{1}{2}\dim H^0(N)\log 2+\frac{1}{4}\zeta'_{0,N}(0,\A_0)+\frac{1}{4}\zeta'_{0,N}(0,-\A_0) +\\+ \frac{1}{8}\textup{Res}\,\zeta_{0,N}(1)\left[\sum_{b=0}^1\left(2x_{1,b}-z_{1,b}(-\A_k)-z_{1,b}(\A_k)\right)\frac{\Gamma'(b+1/2)}{\Gamma (b+1/2)}\right].
\end{align*}
Now we evaluate the combinatorial factor of $\textup{Res}\,\zeta_{0,N}(1)$ by considering the following formulas, encountered in [BGKE, Section 2-3]
\begin{align}
&D_1(t)=\sum_{b=0}^{1}x_{1,b}t^{1+2b}=\frac{1}{8}t-\frac{5}{24}t^3, \nonumber \\ &M_1(t, \A)=\sum_{b=0}^{1}z_{1,b}(\pm \A)t^{1+2b}=\left(-\frac{3}{8}+\A\right )t+\frac{7}{24}t^3.
\label{m-1-a}
\end{align}
Further one needs the following values (calculated from the known properties of Gamma functions)
$$\frac{\Gamma'(1/2)}{\Gamma(1/2)}=-(\gamma+2\log 2), \quad \frac{\Gamma'(3/2)}{\Gamma(3/2)}=2-(\gamma+2\log 2).$$
Finally one observes $\A_0=0$ in this setting. This easily leads to the first formula in the statement of corollary. The second formula follows from the first by $$\zeta_{0,N}(s)=2\zeta_R(s),$$ where the factor $2$ comes from the fact that the eigenvalues $n^2$ of the Laplacian $\triangle_{k=0,S^1}$ are of multiplicity two for $n\neq 0$. The Riemann zeta function has the following special values
$$\zeta'_R(0)=-\frac{1}{2}\log 2 \pi, \quad \textup{Res}\,\zeta_R(1)=1,$$ which gives the second formula.
\end{proof}

\begin{cor}\label{three-dim}
Let $M$ be a three-dimensional bounded generalized cone of length one over a closed oriented manifold $N$ with a metric scaled as in Theorem \ref{final-odd}. Then the analytic torsion of $M$ is given by
\begin{align*}
\log T(M)=\frac{\log 2}{2}\chi(N)-\frac{\log 3}{2}\dim H^0(N)+ \frac{1}{2}\zeta'_{0,N}(0,1/2)- \\ -\frac{1}{2}\zeta'_{0,N}(0,-1/2)+\frac{\log 2}{2}\textup{Res}\, \zeta_{0,N}(1) +\frac{1}{16}\textup{Res}\, \zeta_{0,N}(2).
\end{align*}
\end{cor}
\begin{proof}
In the three-dimensional case the general formula of Theorem \ref{final-odd} reduces to the following expression:
\begin{align*}
\log T(M)=\frac{\log 2}{2}\chi(N)-\frac{\log 3}{2}\dim H^0(N)+ \\+\frac{1}{2}\left(\zeta'_{0,N}(0,\A_0)-\zeta'_{0,N}(0,-\A_0)\right) +\A_0 \textup{Res}\,\zeta_{0,N}(1)\left[\frac{\gamma}{2}+\frac{\Gamma'(1)}{\Gamma(1)}\right] + \\+ \frac{1}{4}\sum_{i=1}^2\textup{Res}\,\zeta_{0,N}(i)\sum_{b=0}^i\left(z_{i,b}(-\A_k)-z_{i,b}(\A_k)\right)\frac{\Gamma'(b+i/2)}{\Gamma (b+i/2)}.
\end{align*}
Now we simply evaluate the last combinatorial sum by considering formulas from [BGKE, (3.6), (3.7)]
\begin{align*}
&M_1(t, \A)=\sum_{b=0}^{1}z_{1,b}(\pm \A)t^{1+2b}=\left(-\frac{3}{8}+\A\right)t+\frac{7}{24}t^3,\\
&M_2(t, \A)=\sum_{b=0}^{2}z_{2,b}(\pm \A)t^{2+2b}=\left(-\frac{3}{16}+\frac{\A}{2}-\frac{\A^2}{2}\right)t^2+\left(\frac{5}{8}-\frac{\A}{2}\right)t^4-\frac{7}{16}t^6.
\end{align*}
We further need the values
\begin{align*}
&\frac{\Gamma'(1)}{\Gamma(1)}=-\gamma, \quad \frac{\Gamma'(1/2)}{\Gamma(1/2)}=-(\gamma+2\log 2),\\ &\frac{\Gamma'(2)}{\Gamma(2)}=1-\gamma, \quad \frac{\Gamma'(3/2)}{\Gamma(3/2)}=2-(\gamma+2\log 2).
\end{align*}
This leads together with $\A_0=1/2$ in the three-dimensional case to the following formula
\begin{align}
\log T(M)=\frac{\log 2}{2}\chi(N)-\frac{\log 3}{2}\dim H^0(N)+ \nonumber \\ +\frac{1}{2}\left(\zeta'_{0,N}(0,1/2)-\zeta'_{0,N}(0,-1/2)\right)-\frac{\gamma}{4}\textup{Res}\, \zeta_{0,N}(1) +\nonumber \\+\frac{1}{4}\left(\textup{Res}\, \zeta_{0,N}(1)[\gamma +2\log 2]+\frac{1}{4}\textup{Res}\, \zeta_{0,N}(2)\right).\label{hilfe-referenz2}
\end{align}
Obvious cancellations in the formula above prove the result.
\end{proof}

\section{Analytic torsion of a cone over $S^1$}\label{an-torsion-sphere} \ \\
\\[-6mm] The preceeding computations reduce in the two-dimensional case simply to the computation of the analytic torsion of a disc. In order to deal with a generalized bounded cone in two dimensions, which is not simply a flat disc, we need to introduce an additional parameter in the Riemannian metric. So in two dimensions the setup is as follows.
\\[3mm] Let $M:=(0,R]\times S^1$ with $$g^M=dx^2\oplus\nu^{-2}x^2g^{S^1}$$ be a bounded generalized cone over $S^1$ of angle arcsec$(\nu)$ and length $1$, with a fixed orientation and with a fixed parameter $\nu\geq 1$. 
\begin{figure}[h]
\begin{center}
\includegraphics[width=0.7\textwidth]{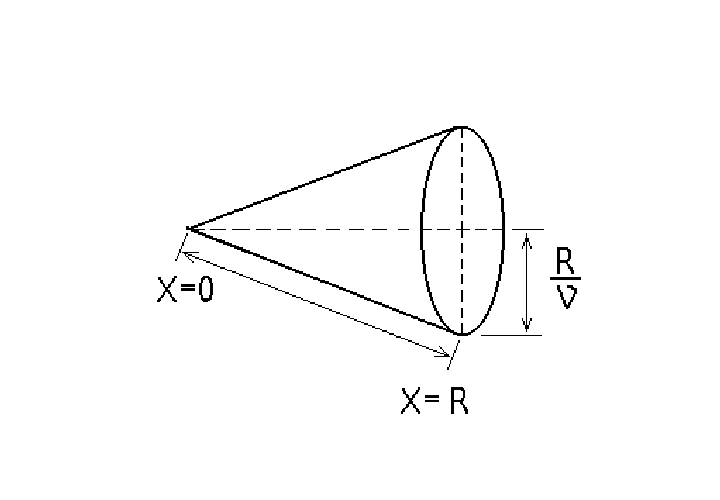}
\caption{A bounded cone of angle arcsec$(\nu), \nu\geq 1$ and length $R$. }
\label{cone-picture}
\end{center}
\end{figure}\ 
\\[-2mm] The main result of our discussion in this part of the presentation is then the following theorem:
\begin{thm}\label{main}
The analytic torsion $T(M)$ of a bounded generalized cone $M$ of length $R$ and angle $\textup{arcsec} \, \nu > 0$ over $S^1$ is given by $$2\log{T(M)}=-\log(\pi R^2) +\log \nu -\frac{1}{\nu}.$$
\end{thm} \ \\
\\[-7mm] This result corresponds precisely to the result obtained in Corollary \ref{two-dim} for the special case $\nu=1$ (for $R=1$). In fact this result can also be derived from [BGKE, Section 5]. This setup was considered by Spreafico in [S]. However [S] deals only with Dirichlet boundary conditions at the cone base. So we extend his approach to the Neumann boundary conditions in order to obtain an overall result for the analytic torsion of this specific cone manifold. We proceed as follows.
\\[3mm] Denote forms with compact support in the interior of $M$ by $\Omega^*_0(M)$. The associated de Rham complex is given by $$0\to \Omega^0_0(M) \xrightarrow{d_0} \Omega^1_0(M)\xrightarrow{d_1}\Omega^2_0(M)\to 0.$$ 
Consider the following maps 
\begin{align*}
\Psi_0:C^{\infty}_0((0,R)&,\Omega^0(S^1))\to \Omega^0_0(M), \\
&\phi \quad \mapsto \quad x^{-1/2}\phi. \\
\Psi_2:C^{\infty}_0((0,R)&,\Omega^1(S^1))\to \Omega^2_0(M), \\
&\phi \quad \mapsto \quad x^{1/2}\phi\wedge dx, \\
\end{align*}
where $\phi$ is identified with its pullback to $M$ under the natural projection $\pi: (0,R]\times N\to N$ onto the second factor, and $x$ is the canonical coordinate on $(0,R]$. We find 
\begin{align*}
\triangle^0:=\Psi_0^{-1}d_0^td_0\Psi_0 = -\frac{d^2}{dx^2}+\frac{1}{x^2}\left(-\nu^2\partial^2_{\theta}-\frac{1}{4}\right) \quad \textup{on} \ C^{\infty}_0((0,R),\Omega^0(S^1)), \\
\triangle^2:=\Psi_2^{-1}d_1d_1^t\Psi_2 = -\frac{d^2}{dx^2}+\frac{1}{x^2}\left(-\nu^2\partial^2_{\theta}-\frac{1}{4}\right) \quad \textup{on} \ C^{\infty}_0((0,R),\Omega^1(S^1)).
\end{align*}
where $\theta$ is the local variable on the one-dimensional sphere. In fact both maps $\Psi_0$ and $\Psi_2$ extend to isometries on the $L^2-$completion of the spaces, by similar arguments as behind Proposition \ref{unitary}. Now consider the minimal extensions $D_k:=d_{k,\min}$ of the boundary operators $d_k$ in the de Rham complex $(\Omega_0^*(M),d)$. This defines by [BL1, Lemma 3.1] a Hilbert complex $$(\mathcal{D},D), \ \textup{with} \ \mathcal{D}^k:=\mathcal{D}(D_k).$$ Put 
\begin{align*}
\triangle^0_{\textup{rel}}:=\Psi_0^{-1}D_0^*D_0\Psi_0, \\
\triangle^2_{\textup{rel}}:=\Psi_2^{-1}D_1D_1^*\Psi_2.
\end{align*}
The Laplacians $\triangle^0_{\textup{rel}}, \triangle^2_{\textup{rel}}$ are spectrally equivalent to $D_0^*D_0, D_1D_1^*$, respectively.
The boundary conditions for $\triangle_{\textup{rel}}^{0}$ and $\triangle_{\textup{rel}}^{2}$ at the cone base $\{1\}\times S^1$ are determined in [BV1, Proposition 4.5].
\\[3mm] In order to identify the boundary conditions for $\triangle_{\textup{rel}}^{0}$ and $\triangle_{\textup{rel}}^{2}$ at the cone singularity, observe that by [BL2, Theorem 3.7] the ideal boundary conditions for the de Rham complex are uniquely determined at the cone singularity. Further [BL2, Lemma 3.1] shows that the corresponding extension coincides with the Friedrich's extension at the cone singularity. We infer from [BS3, Theorem 6.1] that the elements in the domain of the Friedrich's extension are of the asymptotics $O(\sqrt{x})$ as $x\to 0$. Hence we find
\begin{align*}
&\mathcal{D}(\triangle^0_{\textup{rel}})= \\=&\{\phi \in H^2_{loc}((0,R]\times S^1)| \phi(R)=0, \ \phi(x)=O(\sqrt{x}) \ \textup{as} \ x\to 0\},\\
&\mathcal{D}(\triangle^2_{\textup{rel}})= \\= &\{\phi \in H^2_{loc}((0,R]\times S^1)| \phi'(R)-\frac{1}{2R}\phi(R)=0, \ \phi(x)=O(\sqrt{x}) \ \textup{as} \ x\to 0\}.
\end{align*}
The first operator with Dirichlet boundary conditions at the cone base is already elaborated in [S]. We adapt their approach to deal with the second operator with generalized Neumann boundary  conditions at the cone base. The scalar analytic torsion of the bounded generalized cone is then given in terms of both results $$2\log T(M) =\zeta'_{\triangle^2_{\textup{rel}}}(0)-\zeta'_{\triangle^0_{\textup{rel}}}(0).$$
Note that the Laplacian $(-\partial_{\theta}^2)$ on $S^1$ has a discrete spectrum $n^2,n\in \Z$, where the eigenvalues $n^2$ are of multiplicity two, up to the eigenvalue $n^2=0$ of multiplicity one. 
\\[3mm] Consider now a $\mu$-eigenform $\phi$ of $\triangle^2_{\textup{rel}}$. Since eigenforms of $(-\partial^2_{\theta})$ on $S^1$ are smooth, the projection of $\phi$ for any fixed $x\in (0,R]$ onto some $n^2-$eigenspace of $(-\partial^2_{\theta})$ maps again to $H^2_{loc}((0,R]\times S^1)$, still satisfies the boundary conditions for $\dom (\triangle^2_{2,\textup{rel}})$ and hence gives again an eigenform of $\dom (\triangle^2_{2,\textup{rel}})$.
\\[3mm] Hence for the purpose of spectrum computation we can assume without loss of generality the $\mu-$eigenform $\phi$ to lie in a $n^2-$eigenspace of $(-\partial^2_{\theta})$ for any fixed $x\in (0,R]$. This element $\phi$, identified with its scalar part, is a solution to $$-\frac{d^2}{dx^2}\phi(x)+\frac{1}{x^2}\left(\nu^2 n^2-\frac{1}{4}\right)\phi(x)=\mu^2\phi(x),$$ subject to the relative boundary conditions. The general solution to the equation above is $$\phi(x)=c_1\sqrt{x}J_{\nu n}(\mu x) +c_2\sqrt{x}Y_{\nu n}(\mu x),$$ where $J_{\nu n}(z)$ and $Y_{\nu n}(z)$ denote the Bessel functions of first and second kind. The boundary conditions at $x=0$ are given by $\phi(x)=O(\sqrt{x})$ as $x\to 0$ and consequently $c_2=0$. The boundary conditions at the cone base give 
\begin{align*}
\phi'(R)-\frac{1}{2R}\phi(R)= c_1 \mu \sqrt{R}J_{\nu n}'(\mu R)=0.
\end{align*}
Since we are not interested in zero-eigenvalues, the relevant eigenvalues are by Corollary \ref{bessel-zeros} given as follows: $$\lambda_{n,k}=\left(\frac{\widetilde{j}_{\nu n,k}}{R}\right)^2$$ with $\widetilde{j}_{\nu n,k}$ being the positive zeros of $J_{\nu n}'(z)$. We obtain in view of the multiplicities of the $n^2-$eigenvalues of $(-\partial_{\theta}^2)$ on $S^1$ for the zeta-function 
\begin{align*}
\zeta_{\triangle^2_{\textup{rel}}}(s)=\sum_{k=1}^{\infty}\lambda_{0,k}^{-s}+2\!\sum_{n,k=1}^{\infty}\lambda_{n,k}^{-s} =\\ =\sum_{k=1}^{\infty}\left(\frac{\widetilde{j}_{0,k}}{R}\right)^{-2s}\!\!\!\!\!+2R^{2s}\sum_{n,k=1}^{\infty}\widetilde{j}_{\nu n,k}^{-2s}.  
\end{align*}
The derivative at zero for the first summand follows by a direct application of [S, Section 3]: 
\begin{lemma}\label{first-summand}
\begin{align*}
K:=\left.\frac{d}{ds}\right|_{0}\sum_{k=1}^{\infty}\left(\widetilde{j}_{0,k}/R\right)^{-2s}\!=-\frac{1}{2}\log 2\pi -\frac{3}{2}\log R +\log 2.\end{align*}
\end{lemma}
\begin{proof}
The values $\widetilde{j}_{0,k}$ are zeros of $J_0'(z)$. Since $J_0'(z)=-J_1(z)$ they are also zeros of $J_1(z)$. Using [S, Lemma 1 (b)] and its application on [S, p.361] we obtain in the notation therein
\begin{align*}
\left.\frac{d}{ds}\right|_{0}\sum_{k=1}^{\infty}\left(\widetilde{j}_{0,k}/R\right)^{-2s}\!= -B(1)+T(0,1) \\= -\frac{1}{2}\log 2\pi -\frac{3}{2}\log R +\log 2.\end{align*}
\end{proof} \ \\
\\[-8mm] Now we turn to the discussion of the second summand. We put $z(s)=\sum_{n,k=1}^{\infty}\widetilde{j}_{\nu n,k}^{-2s}$ for $Re(s)\gg0$. This series is well-defined for $Re(s)$ sufficiently large by the general result in Theorem \ref{cone-zeta-function}. Due to uniform convergence of integrals and series we obtain with computations similar to \eqref{integral-spreafico} the following integral representation 
\begin{align}\label{int-rep}
z(s)=\frac{s^2}{\Gamma(s+1)}\int_0^{\infty}t^{s-1}\frac{1}{2\pi i}\int_{\wedge_c}\frac{e^{-\lambda t}}{-\lambda}T(s,\lambda)d\lambda dt,
\end{align}
\begin{align}\label{Tt}
T(s,\lambda)=\sum_{n=1}^{\infty}(\nu n)^{-2s}t_n(\lambda), \quad t_n(\lambda)=-\sum_{k=1}^{\infty}\log\left(1-\frac{(\nu n)^2\lambda}{\widetilde{j}_{\nu n,k}^2}\right),
\end{align}
where $\Lambda_c:=\{\lambda \in \C  | |arg(\lambda -c)|=\pi /4\}$ with $c>0$ being any fixed positive number, smaller than the lowest non-zero eigenvalue of $\triangle_{\textup{rel}}^2$. 
\\[3mm] We proceed with explicit calculations by presenting $t_n(\lambda)$ in terms of special functions. Using the infinite product expansion \eqref{prod-Bessel} we obtain the following result for the derivative of the modified Bessel function of first kind:
$$I'_{\nu n}(\nu n z)=\frac{(\nu n z)^{\nu n-1}}{2^{\nu n}\Gamma (\nu n)}\prod_{k=1}^{\infty}\left( 1+\frac{(\nu n z)^2}{\widetilde{j}^2_{\nu n, k}}\right),$$
where $\widetilde{j}_{\nu n, k}$ denotes the positive zeros of $J'_{\nu n}(z)$. Putting $z=\sqrt{-\lambda}$ we get
\begin{align}\nonumber
t_{n}(\lambda)=-\sum_{k=1}^{\infty}\log\left(1-\frac{(\nu n)^2\lambda}{\widetilde{j}_{\nu n,k}^2}\right)= -\log \left[\prod_{k=1}^{\infty}\left(1+\frac{(\nu nz)^2}{\widetilde{j}_{\nu n,k}^2}\right)\right] \\ \label{t-n}= -\log I'_{\nu n}(\nu nz)+\log(\nu nz)^{\nu n-1}-\log 2^{\nu n}\Gamma(\nu n).
\end{align} 
The associated function $T(s,\lambda)$ from \eqref{Tt} is however not analytic at $s=0$. The $1/\nu n$-dependence in $t_{n}(\lambda)$ causes non-analytic behaviour. We put 
\begin{align}\label{p2}
t_{n}(\lambda)=:p_{n}(\lambda)+\frac{1}{\nu n}f(\lambda), \quad P(s,\lambda)=\sum_{n=1}^{\infty}(\nu n)^{-2s}p_{n}(\lambda).
\end{align}
To get explicit expressions for $P(s,\lambda)$ and $f(\lambda)$ we use asymptotic expansion of the Bessel-functions for large order from [O], in analogy to Lemma \ref{f-r}. We obtain in the notation of \eqref{polynom2} with $z=\sqrt{-\lambda}$ and $t=1/\sqrt{1-\lambda}$:
$$f(\lambda)=-M_1(t,0)=\frac{3}{8}t-\frac{7}{24}t^3,$$ where we inferred the explicit form of $M_1(t,0)$ from \eqref{m-1-a}. We obtain for $p_n(\lambda)$
\begin{align}
p_{n}(\lambda)= -\log I'_{\nu n}(\nu nz)+\log(\nu nz)^{\nu n-1}-\log 2^{\nu n}\Gamma(\nu n)-\nonumber \\ -\frac{1}{\nu n}\left(\frac{3}{8}t-\frac{7}{24}t^3\right).\label{p-n}
\end{align}
As in Lemma \ref{ff} we compute the contribution coming from $f(\lambda)$. 
\begin{lemma}\label{f2}
\begin{align*}
\int_0^{\infty}t^{s-1}\frac{1}{2\pi i}\int_{\wedge_c}\frac{e^{-\lambda t}}{-\lambda}f(\lambda)d\lambda dt=\frac{1}{12\sqrt{\pi}}\Gamma\left(s+\frac{1}{2}\right)\left(\frac{1}{s}-7\right).
\end{align*}
\end{lemma}
\begin{proof}
Observe from [GRA, 8.353.3] by substituting the new variable $x=\lambda-1$
\begin{align*}
\frac{1}{2\pi i}\int_{\wedge_c}\frac{e^{-\lambda t}}{-\lambda}\frac{1}{(1-\lambda)^a}d\lambda= \frac{1}{2\pi i}e^{-t}\int_{\wedge_{c-1}}-\frac{e^{-x t}}{x+1}\frac{1}{(-x)^{a}}dx =\\=\frac{1}{\pi}\sin(\pi a)\Gamma(1-a)\Gamma(a,t).
\end{align*}
Using now the relation between the incomplete Gamma function and the probability integral
$$\int_0^{\infty}t^{s-1}\Gamma(a,t)dt=\frac{\Gamma(s+a)}{s}$$ we finally obtain
\begin{align*}
&\int_0^{\infty}t^{s-1}\frac{1}{2\pi i}\int_{\wedge_c}\frac{e^{-\lambda t}}{-\lambda}f(\lambda)d\lambda dt \\  = &\frac{3}{8\pi}\sin\left(\frac{\pi}{ 2}\right)\Gamma\left(1-\frac{1}{2}\right)\frac{\Gamma\left(s+1/2\right)}{s} - \\ -& \frac{7}{24\pi}\sin \left(\frac{3\pi}{2}\right)\Gamma\left(1-\frac{3}{2}\right)\frac{\Gamma\left(s+3/2\right)}{s}\\  = &\frac{3}{8\sqrt{\pi}}\frac{\Gamma\left(s+1/2\right)}{s}- \frac{7}{12\sqrt{\pi}}\frac{\Gamma\left(s+3/2\right)}{s} = \\ =&
\frac{1}{\sqrt{\pi}}\Gamma\left(s+\frac{1}{2}\right)\left\{\frac{3}{8s}+\frac{7}{12s}\left(s+\frac{1}{2}\right)\right\} =\\  = & \frac{1}{12\sqrt{\pi}}\Gamma\left(s+\frac{1}{2}\right)\left(\frac{1}{s}-7\right) .
\end{align*}
\end{proof} \ \\
\\[-8mm] By classical asymptotics of Bessel functions for large arguments and fixed order $$I'_{\nu n}(\nu n z)=\frac{e^{\nu n z}}{\sqrt{2\pi \nu n z}}\left(1+O\left(\frac{1}{z}\right)\right),$$ where the region of validity is preserved (see the discussion in the higher-dimensional case in Proposition \ref{AB}), we obtain for $p_{n}(\lambda)$ from \eqref{p-n}
\begin{align*}
p_{n}(\lambda)=-\nu n\sqrt{\lambda}+\left(\frac{1}{4}+(\nu n-1)\frac{1}{2}\right)\log(-\lambda)+\frac{1}{2}\log 2\pi \nu n \\ +(\nu n-1)\log \nu n -\log (2^{\nu n}\Gamma (\nu n)) + O((-\lambda)^{-1/2}).
\end{align*}
Following [S, Section 4.2] we reorder the summands in the above expression to get $$p_{n}(\lambda) = -\nu n\sqrt{\lambda} +a_{n} \log (-\lambda) + b_{n} + O((-\lambda)^{-1/2}), $$ where the interesting terms are clear from above. We set
\begin{align*}
A(s):=\sum_{n=1}^{\infty}(\nu n)^{-2s}a_{n}=& \frac{1}{2}\nu^{-2s+1}\zeta_R(2s-1)-\frac{1}{4}\nu^{-2s}\zeta_R(2s), \\
B(s):=\sum_{n=1}^{\infty}(\nu n)^{-2s}b_{n}=& \frac{1}{2}\nu^{-2s}\log \left(\frac{2\pi}{\nu}\right) \zeta_R(2s)+\\+&\nu^{-2s+1}\log \left(\frac{\nu}{2}\right)\zeta_R(2s-1)-\nu^{-2s+1}\zeta_R'(2s-1)+ \\ +&\frac{1}{2}\nu^{-2s}\zeta_R'(2s)-\sum_{n=1}^{\infty}(\nu n)^{-2s}\log \Gamma (\nu n). 
\end{align*}
Following the approach of M. Spreafico it remains to evaluate $P(s,0)$ defined in \eqref{p2} in order to obtain a closed expression for the function $z(s)$. 
\begin{lemma}
\begin{align*}
P(s, 0) = -\frac{1}{12}\nu^{-2s-1} \zeta_R(2s+1).
\end{align*}
\end{lemma}
\begin{proof}
Recall the asymptotic behaviour of Bessel functions of second order for small arguments
$$I_{\nu n}(x)\sim \frac{1}{\Gamma (\nu n+1)}\left( \frac{x}{2} \right)^{\nu n} \ \Rightarrow I'_{\nu n}(x)\sim \frac{\nu n}{2\Gamma (\nu n+1)}\left( \frac{x}{2} \right)^{\nu n-1}.$$ Further observe that as $\lambda \to 0$ we obtain with $z=\sqrt{-\lambda}$ and $t=1/\sqrt{1+z^2}$ $$M_1(t,0)= -\frac{3}{8}t+\frac{7}{24}t^3 \xrightarrow{\lambda \to 0} -\frac{3}{8}+\frac{7}{24}=-\frac{1}{12}.$$ Using these two facts we obtain from \eqref{p-n} for $p_{n}(0)$ 
\begin{align*}
p_{n}(0) = -\log \nu n + \log \Gamma (\nu n+1) -\log \Gamma (\nu n) -\frac{1}{12 \nu n} =-\frac{1}{12 \nu n}\\
\Rightarrow P(s, 0) = \sum_{n=1}^{\infty}(\nu n)^{-2s}p_{n}(0)=-\frac{1}{12}\nu^{-2s-1} \zeta_R(2s+1).
\end{align*}
\end{proof} \ \\
\\[-8mm] Now we have all the ingredients together, since by [S, p. 366] and Lemma \ref{f2} the function $z(s)$ is given as follows:
\begin{align*}
&z(s) = \frac{s}{\Gamma (s+1)}[\gamma A(s)-B(s)-\frac{1}{s}A(s)+P(s,0)]\ + \\ &+ \  \frac{s^2}{\Gamma(s+1)}\nu^{-2s-1}\zeta_R(2s+1)\frac{1}{12\sqrt{\pi}}\Gamma\left(s+\frac{1}{2}\right)\left(\frac{1}{s}-7\right) + \frac{s^2}{\Gamma (s+1)}h(s),
\end{align*}
where the last term vanishes with its derivative at $s=0$. We are interested in the value of the function itself $z(0)$ and its derivative $z'(0)$. In order to compute the value of $z(0)$ recall the fact that close to $1$ the Riemann zeta function behaves as follows $$\zeta_R(2s+1)=\frac{1}{2s}+\gamma +o(s), \quad s\to 0.$$ This implies
$$\frac{s^2}{\Gamma(s+1)}\nu^{-2s-1}\zeta_R(2s+1)\frac{1}{12\sqrt{\pi}}\Gamma\left(s+\frac{1}{2}\right)\left(\frac{1}{s}-7\right)\rightarrow \frac{1}{24\nu}, s\to 0.$$
Furthermore note that the function
$$\eta(s,\nu):=\sum_{n=1}^{\infty}(\nu n)^{-2s}\log \Gamma (\nu n +1)-\frac{1}{12}\nu^{-2s-1}\zeta_R(2s+1),$$ introduced in [S, p.366] is regular at $s=0$, cf. [S, Section 4.3]. Hence $\gamma A(s)-B(s)+P(s,0)$ is regular at $s=0$ and we obtain straightforwardly:
\begin{align*}
z(0)=-A(0)+\frac{1}{24\nu}= -\frac{1}{2}\nu \zeta_R(-1)+\frac{1}{4}\zeta_R(0)+\frac{1}{24\nu}.
\end{align*}
In view of the explict values $\zeta_R(-1)=-\frac{1}{12}$ and $\zeta_R(0)=-\frac{1}{2}$ we find 
\begin{align}\label{z(0)}
z(0)=\frac{\nu}{24}+\frac{1}{24\nu}-\frac{1}{8}.
\end{align}
\begin{lemma}\label{dz}
\begin{align*}
z'(0)=\eta(0,\nu)+\frac{1}{2}\log \nu -\frac{1}{4}\log 2\pi -\frac{1}{12}\nu\log 2 + \frac{1}{12\nu}(\gamma -\log 2\nu -\frac{7}{2}),
\end{align*}
where $\eta(s,\nu)=\sum_{n=1}^{\infty}(\nu n)^{-2s}\log \Gamma (\nu n+1)-\frac{1}{12}\nu^{-2s-1}\zeta_R(2s+1)$.
\end{lemma}
\begin{proof}
We compute $z'(0)$ from the above expression for $z(s)$, using $$\Gamma'(1/2)=-\sqrt{\pi}(\gamma +2\log 2).$$ Straightforward computations lead to:
\begin{align}
z'(0)=P(0,0)-A'(0)-B(0)+\frac{1}{12\nu}(\gamma -\log 2\nu -\frac{7}{2}).
\end{align}
The statement follows with $\eta(s,\nu)$ being defined precisely as in [S, Section 4.2].
\end{proof} \ \\
\\[-8mm] Now we are able to provide a result for the derivative of the zeta function $\zeta'_{\triangle^2_{\textup{rel}}}(0)$. Recall $$\zeta_{\triangle^2_{\textup{rel}}}(s)=\sum_{k=1}^{\infty}\left(\frac{\widetilde{j}_{0,k}}{R}\right)^{-2s}\!\!\!\!\!+2R^{2s}\sum_{\nu n,k=1}^{\infty}\widetilde{j}_{\nu n,k}^{-2s}. $$ With $K$ defined in Lemma \ref{first-summand} and $z(s)=\sum_{n,k=1}^{\infty}\widetilde{j}_{\nu n,k}^{-2s}$ we get 
\begin{align*}
\zeta'_{\triangle^2_{\textup{rel}}}(0)= K+4z(0)\log R +2z'(0).
\end{align*}
It remains to compare each summand to the corresponding results for $\zeta'_{\triangle^0_{\textup{rel}}}(0)$ obtained in [S]. Using Lemma \ref{first-summand}, \eqref{z(0)} and \eqref{dz} we finally arrive after several cancellations at Theorem \ref{main} $$2\log T(M) =\zeta'_{\triangle^2_{\textup{rel}}}(0)-\zeta'_{\triangle^0_{\textup{rel}}}(0)=-\log(\pi R^2) +\log \nu -\frac{1}{\nu}.$$ 

\section{References}\
\\[-1mm] [AS] Editors M. Abramowitz, I.A. Stegun \emph{"Handbook of math. Functions"} AMS. 55.
\\[3mm] [BGKE] M. Bordag, B. Geyer, K. Kirsten, E. Elizalde \emph{"Zeta function determinant of the Laplace Operator on the D-dimensional ball"}, Comm. Math. Phys. 179, no. 1, 215-234, (1996)
\\[3mm] [BKD] M. Bordag, K. Kirsten, J.S. Dowker \emph{"Heat-kernels and functional determinants on the generalized cone"}, Comm. Math. Phys. 182, 371-394, (1996)
\\[3mm] [BL1] J. Br\"{u}ning, M. Lesch \emph{"Hilbert complexes"}, J. Funct. Anal. 108, 88-132 (1992)
\\[3mm] [BL2] J. Br\"{u}ning, M. Lesch \emph{"K\"{a}hler-Hodge Theory for conformal complex cones"}, Geom. Funct. Anal. 3, 439-473 (1993)
\\[3mm] [BM] J. Br\"{u}ning, X. Ma \emph{"An anomaly-formula for Ray-Singer metrics on manifolds with boundary"}, Geom. Funct. An. 16, No. 4, 767-837, (2006)
\\[3mm] [Br] J. Br\"{u}ning \emph{$L^2$-index theorems for certain complete manifolds}, J. Diff. Geom. 32 491-532 (1990)
\\[3mm] [BS] J. Br\"{u}ning, R. Seeley \emph{"An index theorem for first order regular singular operators"}, Amer. J. Math 110, 659-714, (1988)
\\[3mm] [BS2] J. Br\"{u}ning, R. Seeley \emph{"Regular Singular Asymptotics"}, Adv. Math. 58, 133-148 (1985)
\\[3mm] [BS3] J. Br\"{u}ning, R. Seeley \emph{"The resolvent expansion for second order Regular Singular Operators"}, J. Funct. Anal. 73, 369-429 (1987)
\\[3mm] [BV1] B. Vertman \emph{"Functional Determinants for Regular-Singular Laplace-type Operators"}, preprint, arXiv:0808.0443 (2008)
\\[3mm] [BZ] J. -M. Bismut and W. Zhang \emph{"Milnor and Ray-Singer metrics on the equivariant determinant of a flat vector bundle"}, Geom. and Funct. Analysis 4, No.2, 136-212 (1994)
\\[3mm] [BZ1] J. -M. Bismut and W. Zhang \emph{"An extension of a Theorem by Cheeger and M\"{u}ller"}, Asterisque, 205, SMF, Paris (1992)
\\[3mm] [C] C. Callias \emph{"The resolvent and the heat kernel for some singular boundary problems"}, Comm. Part. Diff. Eq. 13, no. 9, 1113-1155 (1988)
\\[3mm] [Ch] J. Cheeger \emph{"Analytic Torsion and Reidemeister Torsion"}, Proc. Nat. Acad. Sci. USA 74 (1977), 2651-2654
\\[3mm] [Ch1] J. Cheeger \emph{On the spectral geometry of spaces with conical singularities}, Proc. Nat. Acad. Sci. 76, 2103-2106 (1979)
\\[3mm] [Ch2] J. Cheeger \emph{"Spectral Geometry of singular Riemannian spaces"}, J. Diff. Geom. 18, 575-657 (1983)
\\[3mm] [Dar] A. Dar \emph{"Intersection R-torsion and analytic torsion for pseudo-manifolds"} Math. Z. 194, 193-216 (1987)
\\[3mm] [DK] J.S. Dowker and K. Kisten \emph{"Spinors and forms on the ball and the generalized cone"}, Comm. Anal. Geom. Volume 7, Number 3, 641-679, (1999)
\\[3mm] [DK1] J.S. Dowker and K. Kisten \emph{"Spinors and forms on the ball and the generalized cone"}, Comm. Anal. Geom. Volume 7, Number 3, 641-679, (1999)
\\[3mm] [Do] J. Dodziuk \emph{"Finite-difference approach to the Hodge theory of harmonic forms"}, Amer. J. Math. 98, no. 1, 79-104, (1976)
\\[3mm] [Fr] W. Franz \emph{"\"{U}ber die Torsion einer \"{U}berdeckung"}, J. reine angew. Math. 173, 245-254 (1935)
\\[3mm] [GRA] I.S: Gradsteyn, I.M Ryzhik, Alan Jeffrey \emph{"Table of integrals, Series and Products"}, 5th edition, Academic Press, Inc. (1994)
\\[3mm] [H] S. W. Hawking \emph{"Zeta-function regularization of path integrals in curved space time"}, CM 55, 133-148 (1977)
\\[3mm] [L] M. Lesch \emph{"Determinants of regular singular Sturm-Liouville operators"}, Math. Nachr. (1995)
\\[3mm] [L1] M. Lesch \emph{"Operators of Fuchs Type, Conical singularities, and Asymptotic Methods"}, Teubner, Band 136 (1997)
\\[3mm] [L3] M. Lesch \emph{"The analytic torsion of the model cone"}, Columbus University (1994), unpublished notes.
\\[3mm] [LR] J. Lott, M. Rothenberg \emph{"Analytic torsion for group actions"} J. Diff. Geom. 34, 431-481 (1991)
\\[3mm] [L\"{u}] W. L\"{u}ck \emph{"Analytic and topological torsion for manifolds with boundary and symmetry"}, J. Diff. Geom. 37, 263-322, (1993)
\\[3mm] [M] E. Mooers \emph{"Heat kernel asymptotics on manifolds with conic singularities"} J. Anal. Math. 78, 1-36 (1999) 
\\[3mm] [Mi] J. Milnor \emph{"Whitehead torsion"}, Bull. Ams. 72, 358-426 (1966)
\\[3mm] [Mu] W. M\"{u}ller \emph{"Analytic torsion and R-torsion for unimodular representations"} J. Amer. Math. Soc., Volume 6, Number 3, 721-753 (1993) 
\\[3mm] [Mu1] W. M\"{u}ller \emph{"Analytic Torsion and R-Torsion of Riemannian manifolds"} Adv. Math. 28, 233-305 (1978)
\\[3mm] [Nic] L.I. Nicolaescu \emph{"The Reidemeister torsion of 3-manifolds"}, de Gruyter Studies in Mathematics, vol. 30, Berlin (2003)
\\[3mm] [O] F.W. Olver \emph{"Asymptotics and special functions"} AKP Classics (1987)
\\[3mm] [P] L. Paquet \emph{"Probl'emes mixtes pour le syst'eme de Maxwell"}, Annales Facultè des Sciences Toulouse, Volume IV, 103-141 (1982)
\\[3mm] [Re1] K. Reidemeister \emph{"Die Klassifikation der Linsenr\"{a}ume"}, Abhandl. Math. Sem. Hamburg 11, 102-109 (1935)
\\[3mm] [Re2] K. Reidemeister \emph{"\"{U}berdeckungen von Komplexen"}, J. reine angew. Math. 173, 164-173 (1935)
\\[3mm] [Rh] G. de Rham \emph{"Complexes a automorphismes et homeomorphie differentiable"}, Ann. Inst. Fourier 2, 51-67 (1950)
\\[3mm] [RS] D.B. Ray and I.M. Singer \emph{"R-Torsion and the Laplacian on Riemannian manifolds"}, Adv. Math. 7, 145-210 (1971)
\\[3mm] [S] M. Spreafico \emph{"Zeta function and regularized determinant on a disc and on a cone"}, J. Geom. Phys. 54 355-371 (2005)
\\[3mm] [V] S. Vishik \emph{"Generalized Ray-Singer Conjecture I. A manifold with smooth boundary"}, Comm. Math. Phys. 167, 1-102 (1995)
\\[3mm] [W] J. Weidmann \emph{"Spectral theory of ordinary differential equations"}, Springer-Verlag Berlin-Heidelberg, Lecture Notes in Math. 1258, (1987)
\\[3mm] [W2] J. Weidmann \emph{"Linear Operators in Hilbert spaces"}, Springer-Verlag, New York, (1980)
\\[3mm] [Wh] J. H. Whitehead \emph{"Simple homotopy types"}, Amer. J. Math. 72, 1-57 (1950)
\\[3mm] [WT] G.N. Watson \emph{"A treatise on the theory of Bessel functions"}, Camb. Univ. Press (1922)

\end{document}